\documentclass{aptpub}
\authornames{F.~CHEN {\it et al}.} % insert the authors here for use in running head. If three or more authors please use (for example) M.~YARROW {\it et al}. Author names should follow the same M.~YARROW format and if two authors, separate by 'AND'.
\shorttitle{Optimal stopping time on SMPs} % insert short title here for use in running head

\usepackage{amssymb}
\usepackage{amsmath}
\usepackage{amsfonts}
\usepackage{mathrsfs}
\usepackage{color, dsfont}
\usepackage{bm}
\usepackage{verbatim}
\usepackage{txfonts, graphicx, nicefrac}
\allowdisplaybreaks

    \def\dz{\delta}

\def\lz{s} 
     \def\oz{\omega}
\def\pz{\pi}

        \def\sz{\sigma}
\def\tz{\tau}        
\def\vz{\varepsilon}

\def\ggz{\Gamma}  \def\ooz{\Omega}
\def\ppz{\Pi}

   \def\bq{{\mathscr{B}}}
   
   \def\fq{{\mathscr{F}}}

\def\qd{\quad}
\def\qqd{\qquad}

\def\lt{\left}
\def\rt{\right}

\def\PP{\mathbb{P}}
\def\EE{\mathbb{E}}

\def\leq{\leqslant}
\def\geq{\geqslant}

\def\prf{\medskip \noindent {\bf Proof}. }
\def\deprf{\quad $\square$ \medskip}

\def\be{\begin{equation}}
\def\de{\end{equation}}

\def\dear{\end{eqnarray*}}
\def\lb{\label}

\def\den{\end{enumerate}}

\def\d{\mathrm{d}}

\def\Bo{\mathcal{B}}

\def\Po{\mathcal{P}}

\def\Bo{\mathcal{B}}

\def\PP{\mathbb{P}}
\def\EE{\mathbb{E}}
\def\MM{\mathbb{M}}
\def\RR{\mathbb{R}}
\def\NN{\mathbb{N}}

\def\FF{\mathbb{F}}
\def\TT{\mathbb{T}}

\def\one{\mathds{1}}

\begin{document}%\recd{}{}%Do not alter this line.

\title{Optimal stopping time on semi-Markov processes \newline with finite horizon} % insert title - use \\ if it requires more than one line.

\authorone[Sun Yat-Sen University]{Fang Chen}
\authorone[Sun Yat-Sen University]{Xianping Guo}
\authortwo[Beijing Normal University at Zhuhai]{Zhong-Wei Liao}
%Affiliation is just the name of your university or institution, for example 'University of Sheffield'. Author names should be of the form 'Mark Yarrow'.
%Authors should be ordered alphabetically subject to the convention in that particular authors country. For example 'Remco van der Hofstad' would be listed under 'H' as is standard in the Netherlands.

%Please use the following format for addresses and emails. The APT office will sort this out after you submit your files.
\addressone{School of Mathematics, Sun Yat-Sen University, Guangzhou 510275, China.} % Your postal address goes here.
%\emailone{chenf76@mail2.sysu.edu.cn} %Authors email goes here.
%\emailone{mcsgxp@mail. sysu.edu.cn} %Authors email goes here.

\addresstwo{Corresponding author. College of Education for the Future, Beijing Normal University at Zhuhai, Zhuhai 519087, China.} % Your postal address goes here.
\emailtwo{zhwliao@hotmail.com \& zhwliao@bnu.edu.cn} %Authors email goes here.

\begin{abstract}
% text of abstract goes here!
In this paper, we consider the optimal stopping problem on semi-Markov processes (SMPs) with finite horizon, and aim to establish the existence and computation of optimal stopping times. To achieve the goal, we first develop the main results of finite horizon semi-Markov decision processes (SMDPs) to the case with {\it additional} terminal costs, introduce an explicit construction of SMDPs, and prove the equivalence between the optimal stopping   problems on SMPs and SMDPs. Then, using the equivalence and the results on SMDPs developed here, we not only show the existence of optimal stopping time of SMPs, but also provide an algorithm for computing optimal stopping time on SMPs. Moreover, we show that the optimal and $\varepsilon$-optimal stopping time can be characterized by the hitting time of some special sets, respectively.
\end{abstract}

%insert keywords separated by a semicolon. You should avoid including keywords which also appear in the title.

% insert the primary 2020 Maths Subject Classification number in the first bracket
		% and the secondary ams number(s) in the second bracket
		% e.g. \ams{60E20}{49G03;49F10}
		%Maximum of three in each, ideally one or two in each primary and secondary.
		%codes found here ``https://mathscinet.ams.org/msnhtml/msc2020.pdf''

\section{Introduction} % Initial capital letter, then lower case. No full stop.

The optimal stopping problem is an important branch of the intersection of probability and control theory, which aims to find the optimal stopping time of stochastic systems under some certain criterion, and has been studied and widely applied in finance, such as the pricing of American options and the buying-selling problem, see the monographes \cite{BR11,CRS91,S78} and the references therein. In most existing literature on optimal stopping problems, the case on discrete-time Markov processes (de Saporta et al. \cite{DDN17}; Dufour \& Piunovskiy \cite{DP10}; Huang \& Zhou \cite{HZ19}) and the case on continuous-time Markov processes (Arkin \& Slastnikov \cite{AS21}; B\"auerle \& Popp \cite{BP18}; Christensen  \& Lindensj\"{o} \cite{CL20};  Gapeev et al. \cite{GKL21}; Shao \& Tian \cite{ST21}) are commonly considered.

Note that between two jump epochs in discrete-time Markov processes  and continuous-time Markov processes are constant and  exponentially distributed, respectively. As is well known, semi-Markov processes (SMPs) are more general than these two processes since the time between two jump epochs in SMPs can not only follow more distribution but also depend on the present state and the next state. Thus, SMPs haven been widely studied and applied in many areas \cite{HG11,LO01,JN06}. However, to the best of our knowledge, the optimal stopping problem on SMPs is addressed only in Boshuizen \& Gouweleeuw \cite{BG93} and Chen et al. \cite{CGL21}. More precisely, Boshuizen \& Gouweleeuw \cite{BG93} and Chen et al. \cite{CGL21} both studied the optimal stopping problems under the infinite-horizon discounted criterion, and show the existence of an optimal time for time-dependent costs in Boshuizen \& Gouweleeuw \cite{BG93} and state-dependent costs in Chen et al. \cite{CGL21}, respectively. As in well-known in many practical applications, the underlying process will be forced to end at a certain time. Thus, it is natural and desirable to consider the optimal stopping problem on SMPs with finite horizon, which  has not been studied yet, and which will be studied in this paper.

In practical applications, if one breaks a contract, one has to pay an additional penalty, which will be regarded as a terminal cost. Thus, in optimal stopping problems with finite horizon, the costs consist of both the running costs and  additional  terminal costs \cite{BR11}. To deal with the optimal stopping problem on SMPs with finite horizon, we follow the idea of transforming the discrete-time optimal stopping problem into an equivalent discrete-time Markov decision process in \cite{BR11}, which has an obvious  advantage that the existence and computation of optimal stopping times can be obtained together.  Indeed, noting that ``continue" or ``stop" may be considered as an action, using the date of the optimal stopping problem on SMPs, we also construct a corresponding SMDP with the action space $\{0,1\}$, where the action 0 and action 1 mean continuation and stop respectively, and prove the  equivalence between the optimal stopping time problems on SMPs and the corresponding SMDPs, that is, given any deterministic policy in the SMDPs and planning horizon, we can induce a stopping time with the same expected cost as that for the policy, and vice versa. To deal with optimal stopping problem on SMPs with a terminal cost by the equivalent SMDP, we need to extend the results in \cite{HG11} without any terminal cost to the more natural case with an additional terminal cost, and establish the existence of an optimal policy and an approximation algorithm for the value function of the SMDP by a minimal non-negative solution method.  Using this equivalence and the results about SMDPs developed here, we not only show the existence of optimal stopping time on SMPs, but also provide an algorithm for computing optimal stopping time on SMPs. Moreover, we show that the optimal and $\varepsilon$-optimal stopping time can be characterized by the hitting time of some special sets.

The rest of the paper is organized as follows. We describe optimal stopping problems on SMPs with finite horizon in Section 2. In section 3, we develop some results on SMDPs with additional terminal costs. Our main results on the existences and computation of optimal stopping times  are given in Section 4 after giving the preliminaries in Section 3.

\section{Optimal stopping problem on semi-Markov processes}

{\bf Notation.} If $X$ is a Borel space, we denote by $\Bo(X)$ the Borel $\sigma$-algebra, by $\Po(X)$ the set of all probability measures on $\Bo(X)$, by $\delta_x$ the Dirac measure at the point $x$, and by $\one_{D}$ the indicator function on the set $D\subset X$. Moreover, let $\RR:=(-\infty,+\infty)$, $\RR_+:=[0,+\infty)$, $x^+:=\max\{x,0\}$ and
$x\wedge y:=\min\{x,y\}$ (for all $x,y\in \RR$). Finally, for any sequence $\{x_k\}\subset R$, we use the convention $\sum_{k=n}^m y_k=0$ if $n>m$.

The model of SMPs is the two-tuples as below
\be\lb{smp}
\{E,Q(\cdot,\cdot|x)\}
\de
where $E$ is the state space, which is assumed to be a Borel space and the transition mechanism of the SMPs is defined by the semi-Markov kernel $Q(\cdot,\cdot|x)$ on $\RR_+\times E$ given $E$. It is assumed that:
\begin{itemize}
\item[(i)] given any $B\in \Bo(E)$ and $x\in E$, $Q(\cdot, B| x)$ is a non-decreasing right continuous real-valued function on $\RR_+$, with $Q (0, B| x) =0$;
\item[(ii)] given any $t \in \RR_+$, $Q (t, \cdot | \cdot)$ is a sub-stochastic kernel on $E$;
\item [(iii)] $\lim_{t \to \infty} Q (t, \cdot | \cdot)$ is a stochastic kernel on $E$.
	\end{itemize}
Then, we introduce the measurable space $(\ooz, \fq)$, which is based on the Kitaev construction (see \cite{K86, KR95}),
$$
\ooz = \lt\{ (x_0, t_1, x_1, \ldots, t_n, x_n, \ldots) : x_0\in E, (t_n,x_n) \in \RR_+ \times E, n \geq 1 \rt\},
$$
and $\fq$ is the corresponding product Borel $\sz$-algebra.
The history of SMPs up to the $n$-th jump epoch is
$$ 
h_0 = x_0, \qd h_{n+1} = (x_0, t_1, x_1, \ldots, t_{n+1}, x_{n+1}), \qd n \geq 0.
$$ 
Let $H_n$ be the set of all histories $h_n$. For each $\oz=(x_0,t_1,\ldots,x_n,t_{n+1},\ldots)\in \ooz$ , define
\begin{align*}
&X_n(\oz)=x_n,\qd T_0(\oz)=0, \qd T_{n+1}(\oz)=t_{n+1},\qd S_n(\omega)=\sum_{k=0}^n T_{k}(\omega), \qd \forall n\geq 0,
\end{align*}
where $S_n$, $T_{n+1}$ and $X_n$ denote the $n$-th jump time, the sojourn time between the $n$-th and $(n+1)$-th jumps and the state at the $n$-th jump time, respectively. Further, we assume here that the decision may only depend on the observation of the marked point process $\{T_n,X_n,n\geq 0\}$. Thus we denote by $\fq_n$ the filtration generated by $\{T_n,X_n,n\geq 0\}$, i.e.,
\begin{align*}
\fq_n:=\sigma(T_0,X_0,\ldots,T_n,X_n).
\end{align*}
Hence, we can give the definition of stopping times as following.
\begin{defn}\lb{s3-defn1}
A random variable $\tz: \ooz \to \mathbb{N} \cup \{+\infty\}$ is called $\fq_n$-stopping time if for all $n\in \mathbb{N}$,
 $$
	\{  \tz  = n \} \in \fq_n .
$$
\end{defn}
This condition means that upon observing the marked point process $\{X_n,T_n,n\geq 0\}$ until $n$-th jump time we can decide whether or not $\tz$ has already occurred. Since the filtration will always be generated by $\{X_n,T_n,n\geq 0\}$ in this paper, we will not mention it explicitly. Denote by $\ggz$ the set of all stopping times.

Using the Tulcea theorem (see \cite[Proposition C.10]{HL96}), for each $x\in E$, there exists a unique probability measure $\PP_x$ on $(\ooz, \fq)$ satisfying that $\PP_x ( T_0 =0, X_0 =x ) =1$ and
\begin{align*}
& \PP_x (T_{n+1} \leq t, X_{n+1}\in B | Y_n ) = Q( t, B | X_n ), % \lb{pp-2}
\end{align*}
where $Y_n=(X_0, \ldots,T_n, X_n)$. Denote by $\EE_x$ the expectation with respect to $\PP_x$. Moreover, we give the following assumption, which can ensure the regularity of SMPs, i.e., $\PP_x ( \lim_{n \to \infty}S_n=\infty)=1$.
\begin{assumption}\lb{s2-ass1}
There exist $\dz > 0$ and $\epsilon > 0$, such that
\begin{equation}\lb{ass}
Q (\dz, E | x) \leq 1 - \epsilon, \qd \forall x \in E.
\end{equation}
\end{assumption}

\noindent The Assumption \ref{s2-ass1} is a standard regular condition widely used in SMPs and SMDPs, see \cite{CGL21,  HG11,LO01}, for instance. According to \cite{HG11}, the Assumption \ref{s2-ass1} implies that
$$
\PP_x ( \lim_{n \to \infty} S_n = \infty )=1, \qd \forall x \in E.
$$
Corresponding to $\{ (T_n, X_n), n \geq 0\}$, we define an underlying continuous-time state process $\{ X(t), t\in \RR_+ \}$ by
$$
X (t) = X_n, \qd S_n \leq t < S_{n+1}.
$$
Refer to Limnios and Oprisan \cite{LO01} for more details about $\{ X(t), t \in [0, \infty) \}$. Next step, we introduce the optimal stopping problem with finite horizon.

Let $c (x)$ and $g(x)$ be the nonnegative measurable real-valued functions on $E$, which represent the cost rates and the terminal costs, respectively. For a given planning horizon $T\in \RR_+$, the optimal stopping time problem with finite horizon $T$ implies that if we have not stopped before time $T$ we must stop paying at time $T$. Thus, for a given planning horizon $T\in \RR_+$, if we choose a stopping time $\tau\in \ggz$, we pay the cost
\begin{equation}\lb{r-t}
R_{\tau}^T :=
\left\{
\begin{aligned}
&{ \int_{0}^{S_{\tz}}  c(X(t)) \d t + g ( X(S_\tz) ) , } & &S_\tz< T; \\
&{ \int_{0}^{T}  c(X(t)) \d t, } & & S_\tz \geq T. \\
\end{aligned}
\right.
\end{equation}
\begin{rem}
From the definition $R_\tau^T$, if we stop before time $T$, we need to pay the terminal cost at the time $S_\tau$. This is very common in practical applications, such as house renting problems. Because if one breaks the contract, one has to pay an additional penalty.
\end{rem}
The $T$-horizon expected cost of a stopping time $\tau $ is given by
\begin{equation}\lb{st-c}
V^\tau (T,x) := \EE_{x}\lt[ R^T_\tau\rt] ,\qd x\in E.
\end{equation}
Then, the value function of optimal stopping problems with finite horizon $T$ is defined by
\be\lb{st-v}
V^*(T,x):=\inf_{\tau\in \ggz}V^\tau (T,x) .
\de
\begin{defn}
Given any planning horizon $T\in \RR_+ $, a stopping time $\tz^* \in \ggz$ is called $T$-optimal if it satisfies that
$$
V^{\tz^*} (T,x) = V^* (T,x) = \inf_{\tz \in \ggz}V^{\tau} (T,x) ,\qd \forall  x\in E.
$$
\end{defn}

Here and what follows, we fix a planning horizon $T\in \RR_+$. The main purpose of this paper is to find a $T$-optimal stopping time and give an algorithm for computing the value function $V^*$.

\section{On semi-Markov decision processes}
We want to solve the stopping time problem by formulating it as SMDPs, so we need to consider the model of SMDPs with a terminal cost and give some results about the model. Moreover, if the terminal cost is always equal to $0$, these results are same as those about SMDPs without terminal cost in \cite{HG11}.

Here and in what follows, we always use ``$\hat{\qd} $'' to distinguish SMDPs from SMPs. The model of SMDPs is introduced by:
$$
\{\hat{E},A,(A(x),x\in \hat{E}),\hat{Q}(\cdot,\cdot|x,a),\hat{c}(x,a),\hat{g}(x,a)\}
$$
where $\hat{E}$ is the state space and $A$ is the action set, which are assumed to be a Borel space and a denumerable set, respectively; $A(x) \subset A$ denotes the set of admissible actions at $x\in \hat{E}$, which assume to be finite; $\hat{Q}(\cdot,\cdot|x,a)$ is the semi-Markov kernel on $\RR_+\times \hat{E}$ given $K$, where $K:=\{(x,a)|x\in \hat{E},a\in A(x)\}$ denotes the set of admissible state-action pairs. Assume that $K\in \Bo(\hat{E})\times \Bo(A)$ and that there exists a measurable mapping $f:\RR\times \hat{E} \to A$ such that $(x,f(t,x))\in K$ for all $(t,x)\in \RR\times \hat{E}$. Finally, the functions $\hat{c}(x,a)$ and $\hat{g}(x,a)$ on $K$ represent the cost rates and terminal costs, which are assumed to be nonnegative and measurable.

\begin{rem}
If $\hat{E}$ is denumerable and $\hat{g}=0$, the model is same as that in \cite{HG11}.
\end{rem}

The evolution of the finite horizon SMDPs as follows. Initially, the system occupies some state $x_0\in \hat{E}$ and the decision maker has a planning horizon $\lz\in \RR$, then he/she chooses an action $a_0\in A(x_0)$ according to the current state $x_0$ and the planning horizon $s$. As a consequence, the system jumps to state $x_1$ after a sojourn time $t_1$ in $x_0$, in which the transition law is subject to the semi-Markov kernel $\hat{Q}$. At time $t_1$, there is a remaining planning horizon $\lz-t_1$ for the decision maker. According to the current state $x_1$ and the current planning horizon $s-t_1$ as well as the previous state and action $(x_0,a_0)$ and the sojourn time $t_1$, the decision maker chooses an action $a_1\in A(x_1)$ and the same sequence of events occur. The decision process evolves in this way and thus we obtain a remaining planning horizon $s-\sum_{k=1}^n t_k$ and an admissible history $\hat{h}_n$ and of the SMDPs up to the $n$-th decision epoch i.e.,
$$
\hat{h}_n=(x_0,a_0,t_1,x_1,\ldots,a_{n-1},t_n,x_n),
$$
where $(x_m,a_m)\in K$, $t_{m+1}\in \RR_+$ for all $m=0,1,\ldots,n-1$, $x_n\in \hat{E}$. Let $\hat{H}_n$ denote the set of all admissible histories $\hat{h}_n$ of the system up to the $n$-th decision epoch, which is endowed with the corresponding product $\sigma$-algebra.
\begin{defn}
A policy $\pi=\{\pi_n,n\geq 0\}$ is a sequence 	of stochastic kernels $\pi_n$ on $A$ given $\RR\times \hat{H}_n$ satisfying
\be  \lb{pi-def}
\pi_{n}(A(x_n)|s,\hat{h}_n)=1, \qd \forall n\geq 0,\hat{h}_n=(x_0,a_0,t_1,x_1,\ldots,a_{n-1},t_n,x_n)\in \hat{H}_n.
\de
\end{defn}
The set of all policies is denoted by $\ppz$.
\begin{rem}
The $s$ in (\ref{pi-def}) means the remaining planning horizon up to the $n$-th decision epoch, and it is assumed can be negative just for convenience. Moreover, the definition of policies here is horizon-relevant, whereas that in infinite horizon case is not.
\end{rem}
To distinguish the subclasses of $\ppz$, we introduce the following notations.

{\bf Notation.} Let $\Phi$ represent the set of stochastic kernels $\varphi$ on $A$ given $\RR \times \hat{E}$ such that $\varphi(A(x)|s,x)=1$ for all $(s,x)\in \RR \times \hat{E}$, and $\FF$ represent the set of measurable functions $f:\RR\times \hat{E} \to A$ such that $f(s,x)\in A(x)$ for all $(s,x)\in \RR \times \hat{E}$.

\begin{defn}\lb{pclass}
\begin{description}
\item [\rm (a)] A policy $\pi=\{\pi_n\}$ is said to be randomized Markov if there is a sequence $\{\varphi_n\}$ of stochastic kernels $\varphi_n\in \Phi$ such that $\pi_n(\cdot|s,\hat{h}_n)=\varphi_n(\cdot|s,x_n)$ for every $(s,\hat{h}_n)\in \RR \times \hat{H}_n$ and $n\geq 0$. We write such a policy as $\pi=\{\varphi_n\}$.
\item[\rm (b)] A randomized Markov policy $\pi=\{\varphi_n\}$ is said to be randomized stationary if $\varphi_n$ are independent of $n$. In this case, we write $\pi$ as $\varphi$ for simplicity.
\item[\rm (c)] A policy $\pi=\{\pi_n\}$ is called deterministic if there exists a sequence $\{d_n\}$ of measurable functions $d_n:\RR \times \hat{H}_n \to A$ such that for all $(s,\hat{h}_n)\in \RR \times \hat{H}_n$ and $n\geq 0$, $d_n(s,\hat{h}_n)\in A(x_n)$ and $\pi_n(\cdot|s,\hat{h}_n)$ is the Dirac measure at $d_n(s,\hat{h}_n)$, i.e.,
$$\pi_n(a|s,\hat{h}_n)=\delta_{\{d_n(s,\hat{h}_n)\}}(a), \quad \forall \, a\in A.$$
We write such a policy as $\pi=\{d_n\}$.
\item[\rm (d)]A randomized Markov policy $\pi=\{\varphi_n\}$ is said to be deterministic Markov if there is a sequence $\{f_n\}$ of functions $f_n\in \FF$ such that $\varphi_n(\cdot|s,x)$ is concentrated at $f_n(s,x)$ for all $(s,x)\in \RR \times \hat{E}$ and $n\geq 0$. We write such a policy as $\pi=\{f_n\}$.
\item[\rm (e)]A deterministic Markov policy $\pi=\{f_n\}$ is said to be deterministic stationary if $f_n$ are independent of $n$. In this case, we write $\pi$ as $f$ for simplicity.
\end{description}	
\end{defn}

For convenience, we denote by $\Pi_{RM}$, $\Pi_{RS}$, $\Pi_{DH}$, $\Pi_{DM}$ and $\Pi_{DS}$ the families of all randomized Markov, randomized stationary, deterministic, deterministic Markov and deterministic stationary policies, respectively. Obviously, $\FF=\Pi_{DS}\subset \Pi_{DM} \subset \Pi_{DH} \subset \Pi$ and  $\FF \subset \Phi=\Pi_{RS} \subset \Pi_{RM} \subset \Pi$.

Let $\hat{\Omega}=(\hat{E}\times A \times \RR_+)^\infty$ be a sample space and $\hat{\fq}$ be the corresponding product $\sigma$-algebra. Similar to SMPs, for any $\hat{\oz}=(x_0,a_0,t_1,x_1,\ldots,a_{n},t_{n+1},x_{n+1},\ldots)\in \hat{\ooz}$ and $n\geq 0$, we can define
$$
\hat{T}_0(\hat{\oz})=0,\quad \hat{T}_{n+1}(\hat{\oz})=t_{n+1}, \quad \hat{X}_n(\hat{\oz})=x_n, \quad A_n(\hat{\oz})=a_n.
$$
Further, for all $n\geq 0$, let $\hat{S}_n=\sum_{m=0}^n \hat{T}_m$. And then, we define $\{\hat{X}(t), A(t),t\in \RR_+\}$ by
\begin{align*}
\hat{X}(t) :=
\left\{
\begin{array}{ll}
\hat{X}_n , \qd \hat{S}_n \leq t < \hat{S}_{n+1} ; \\
\partial_S, \qd t\geq \lim_{n \to\infty}\hat{S}_n ;
\end{array}
\right.
 \qd A(t)  :=
\left\{
\begin{array}{ll}
A_n ,  \hat{S}_n \leq t < \hat{S}_{n+1} ; \\
\partial_A,  t\geq  \lim_{n \to\infty}\hat{S}_n ,
\end{array}
\right.
\end{align*}
where $\partial_S$ and $\partial_A$ are the extra state and action jointed to $\hat{E}$ and $A$, respectively. Now, given $(s,x)\in \RR\times \hat{E}$ and $\pi=\{\pi_n\}\in \Pi$, by the Ionescu Tulcea theorem (see \cite[Proposition B.2.5]{BR11}), there exists a unique probability measure $\hat{\PP}_{(s,x)}^\pi$ on $(\hat{\Omega},\hat{\fq})$ such that
\begin{align}
& \hat{\PP}_{(s,x)}^\pi ( \hat{T}_0 =0, \hat{X}_0=x ) =1, \lb{dpp-1}\\
& \hat{\PP}_{(s,x)}^\pi  (A_n=a | \hat{Y}_n ) = \pi_{n}(a|(s-\hat{S}_n),\hat{Y}_n), \lb{dpp-2}\\
& \hat{\PP}_{(s,x)}^\pi (\hat{T}_{n+1}\leq t,\hat{X}_{n+1}\in B|\hat{Y}_n,A_n)=\hat{Q}(t,B|\hat{X}_n,A_n), \lb{dpp-3}
\end{align}
where$ \hat{Y}_n:=( \hat{X}_0, A_0, \hat{T}_1, \hat{X}_1,\ldots,  A_{n-1}, \hat{T}_n, \hat{X}_n)$, $t\in \RR_+$, $B\in \Bo(\hat{E})$, $a\in A$ and $n\geq 0$. The expectation operator with respect to $\hat{\PP}_{(s,x)}^\pi$ is denoted by $\hat{\EE}_{(s,x)}^\pi$. Recall that we fix a planning horizon $T$ in Section 2. To treat $T$-horizon optimization problem, naturally, we assume a finite number of jumps until time $T$. Thus, we propose Assumption \ref{ass1}.
\begin{assumption}\lb{ass1}
For all $(\lz,x)\in \RR\times \hat{E}$ and $\pz\in \ppz$, $\hat{\PP}_{(s,x)}^\pz(\{\lim_{n\to \infty}\hat{S}_n=\infty\})=1$.
\end{assumption}

Given $(\lz,x)\in \RR_+ \times \hat{E}$, we define the expected cost of a policy $\pi\in \ppz$ by
$$
U^{\pi}(s,x):=\hat{\EE}^\pi_{(s,x)}\left[  \int_0^{s} \hat{c}(\hat{X}(t),A(t))\d t+ \hat{g}(\hat{X}(s),A(s)) \right]
$$
and the value function (or minimum expected cost) by
$
U^*(s,x):=\inf_{\pi\in \Pi}U^{\pi}(s,x).
$
\begin{defn}
A policy $\pi^*\in \Pi$ is called $T$-horizon optimal if
$$
U^{\pi^*}(s,x)=U^*(s,x),\quad \forall \, (s,x)\in [0,T]\times \hat{E}.
$$
\end{defn}

\begin{rem}
Noting that the stopping time $\tau^*$ is $T$-optimal if and only if it achieves $V^*$ for the fixed planning horizon $T$, while if the policy $\pi^*\in \ppz$ is $T$-optimal, $\pz^*$ needs to achieve the value function of SMDP for all $\lz\in [0,T]$.
\end{rem}

The focus of this section is on finding an optimal policy in $\Pi$, which can deduce the optimal stopping time. The following result reveals that it suffices to seek for optimal policies in $\Pi_{RM}$.

\begin{prop}\lb{smdp-prop1}
Suppose that Assumption \ref{ass1} holds. Then for each $\pz=\{\pz_n\}\in \ppz$ and $(s,x)\in \RR_+\times \hat{E}$, there exists a policy $\hat{\pi}=\{\varphi_n\}\in \Pi_{RM}$ such that $U^\pi(s,x)=U^{\hat{\pi}}(s,x)$.
\end{prop}

\begin{proof}
Under Assumption \ref{ass1}, the monotone convergence theorem gives that
\begin{align}\lb{vn}
U^{\pi}(s,x)
=&\sum_{m=0}^\infty \hat{\EE}_{(s,x)}^\pi \bigg[   ((s -\hat{S}_m)^+\wedge \hat{T}_{m+1})\hat{c}(\hat{X}_m,A_m)+\one_{[\hat{S}_m,\hat{S}_{m}+\hat{T}_{m+1})}(s)\hat{g}(\hat{X}_m,A_m) \bigg].
\end{align}
Hence, it suffices to show that there is a policy $\hat{\pi}=\{\varphi_n\}\in \Pi_{RM}$ such that
$$
\hat{\PP}_{(s,x)}^{\pi}(\hat{X}_n\in B,A_n=a,\hat{S}_n\leq t,\hat{T}_{n+1}\leq v)=\hat{\PP}_{(s,x)}^{\hat{\pi}}(\hat{X}_n\in B,A_n=a,\hat{S}_n\leq t,\hat{T}_{n+1}\leq v),
$$
for $n=0,1,\ldots$, $t,v\in \RR_+$, $B\in\Bo(\hat{E})$, and $a\in A$. Moreover, noting that (\ref{dpp-3}) implies
\begin{align*}
\hat{\EE}_{(s,x)}^{\pi}\lt[ \one_{[0,v]}(\hat{T}_{n+1})  | \hat{X}_n,A_n,\hat{S}_n\rt]= \hat{Q}(v,\hat{E}|\hat{X}_n,A_n),
\end{align*}
we need only to prove that
\be\lb{prop1-1}
\hat{\PP}_{(s,x)}^{\pi}(\hat{X}_n\in B,A_n=a,\hat{S}_n\leq t)=\hat{\PP}_{(s,x)}^{\hat{\pi}}(\hat{X}_n\in B,A_n=a,\hat{S}_n\leq t).
\de
Indeed, fix $(s,x)\in \RR_+\times \hat{E}$, and define a randomized Markov policy $\hat{\pi}:=\{\varphi_n\}$ (depending on $(s,x)$) by
\begin{align}\lb{prop1-2}
\varphi_n(a|t,y) :=
\left\{
\begin{array}{ll}
\hat{\EE}_{(s,x)}^{\pi}[\one_{\{a\}}(A_n)|\hat{S}_n=s-t,\hat{X}_n=y] , & t\leq \lz,y\in \hat{E} ; \\
\frac{1}{|A(y)|}, & t> \lz,y\in\hat{E};
\end{array}
\right.
\end{align}
where $|A(y)|$ is the cardinality of $A(y)$. We show by induction that (\ref{prop1-1}) holds with $\hat{\pi}$ defined through (\ref{prop1-2}). Clearly it holds with $n=0$. Assume that (\ref{prop1-1}) holds for some $n$ ($n\geq 0$). Then,
\begin{align}\lb{prop1-3}
&\hat{\PP}_{(s,x)}^{\pi}(\hat{X}_{n+1}\in B, \hat{S}_{n+1}\leq t) \notag\\
 =&\hat{\EE}_{(s,x)}^{\pi}\big[\hat{\EE}_{(s,x)}^{\pi}\big[ \one_{B}(\hat{X}_{n+1})\one_{[0,t]}(\hat{S}_{n+1})\big|\hat{X}_n,A_n,\hat{S}_n\big]\big]\notag\\
=&\hat{\EE}_{(s,x)}^{\pi}\big[\hat{Q}((t-\hat{S}_n)^+,B|\hat{X}_n,A_n)\big] \qd \text{ (by (\ref{dpp-3}))}\notag\\
=&\hat{\PP}_{(s,x)}^{\hat{\pi}}(\hat{X}_{n+1}\in B, \hat{S}_{n+1}\leq t). \qd \text{(by the induction hypothesis})
\end{align}
Therefore, the definition of $\hat{\pz}$ and the above equality give that
\begin{align*}
&\hat{\PP}_{(s,x)}^{\pi}(\hat{X}_{n+1}\in B, \hat{S}_{n+1}\leq t, A_{n+1}=a) \notag\\
=&\hat{\EE}_{(s,x)}^{\pi}\big[ \one_{B}(\hat{X}_{n+1})\one_{[0,t]}(\hat{S}_{n+1}) \hat{\EE}_{(s,x)}^{\pi}[\one_{\{a\}}(A_{n+1})|\hat{S}_{n+1},\hat{X}_{n+1}]\big]  \notag\\
=&\hat{\EE}_{(s,x)}^{\pi}\big[ \one_{B}(\hat{X}_{n+1})\one_{[0,t]}(\hat{S}_{n+1}) \varphi_n(a|s-\hat{S}_{n+1},\hat{X}_{n+1})\big] \qd \text{(by (\ref{prop1-2}))}\notag\\
=&\hat{\EE}_{(s,x)}^{\pi}\big[ \one_{B}(\hat{X}_{n+1})\one_{[0,t]}(\hat{S}_{n+1})  \hat{\EE}_{(s,x)}^{\hat{\pi}}[\one_{\{a\}}(A_{n+1})|\hat{S}_{n+1},\hat{X}_{n+1}]\big] \qd \text{(by (\ref{dpp-2}))}\notag\\
=&\hat{\PP}_{(s,x)}^{\hat{\pi}}(\hat{X}_{n+1}\in B, \hat{S}_{n+1}\leq t, A_{n+1}=a). \qd \text{(by(\ref{prop1-3}))}
\end{align*}
Thus the induction hypothesis is satisfied and the proof is completed.
\end{proof}

Due to Proposition \ref{smdp-prop1}, we limit our discussion to randomized Markov policies in the rest of this section. Next, we establish our main results about the SMDPs. That is, we prove that the value function is a minimum nonnegative solution to the optimality equation and that there exists an optimal deterministic stationary policy. Also, we derive an algorithm for computing optimal policies and the value function.

Let $\MM$ be the set of Borel measurable functions $v:[0,T]\times \hat{E} \to \bar{\RR}_+:=[0,\infty]$. Given any $a\in A$, define an operator $\TT^a$ from $\MM$ to itself as: for each $v\in \MM$ and $x\in \hat{E}$ if $a\notin A(x)$, $\TT^a v(\lz,x):=+\infty$, otherwise
\begin{align*}
\TT^a v(\lz,x):=&\hat{c}(x,a)\int_{0}^{\lz} (1-\hat{Q}(t,\hat{E}|x,a)\d t+ \hat{g}(x,a) (1-\hat{Q}(s,\hat{E}|x,a))\\
&+\int_{[0,s]} \int_{\hat{E}} v(s-t,y) \hat{Q}(\d t, \d y|x,a).
\end{align*}
Moreover, any $v\in \MM$, $\varphi\in \Phi$, and  $(\lz,x)\in [0,T]\times \hat{E}$, let
\begin{align*}
\TT^{\varphi}v(s,x):=\sum_{a\in A(x)}\varphi(a|s,x)\TT^a v(s,x) ,\qd \text{and}\qd\TT v(s,x):=\min_{a\in A(x)}\TT^a v(s,x).
\end{align*}

To establish the iteration algorithm for computing $U^\pi$ and $U^*$, we define a function sequence $\{U^\pi_n\}$ as following:
\begin{align*}
U_{-1}^\pi (s,x)&:=0,\\
U_n^\pi(s,x)&:=\sum_{m=0}^n \hat{\EE}_{(s,x)}^\pi \bigg[   ((s - \hat{S}_m)^+\wedge \hat{T}_{m+1})\hat{c}(\hat{X}_m,A_m)+\one_{[\hat{S}_m,\hat{S}_{m+1})}(s)\hat{g}(\hat{X}_m,A_m) \bigg]
\end{align*}
for every $(s,x)\in[0,T]\times \hat{E}$ and $n\geq 0$. Clearly, $U^\pi_n(s,x)\leq U^\pi_{n+1}(s,x)$ for every $n\geq- 1$, and moreover, it follows from (\ref{vn}) that $\lim_{n \to \infty }U^\pi_n(s,x)=U^\pi(s,x)$.

The following lemma is basic to our results.

\begin{lem}\lb{smdp-lemma1}
Suppose that Assumption \ref{ass1} holds. Let $\pi=\{\varphi_n,n\geq 0\}\in \Pi_{RM}$ be arbitrary.
\begin{itemize}
\item [\rm (a)] For each $n\geq -1$, $U^\pi_{n+1}=\TT^{\varphi_0}U^{^{(1)}\pi}_n$ and $U^\pi=\TT^{\varphi_0}U^{^{(1)}\pi}$, where $^{(1)}\pi=\{\varphi_n,n\geq 1\}$.
\item [\rm (b)] In particular, for each $\varphi\in \Phi$, $U^{\varphi}_{n+1}=\TT^{\varphi}U_n^{\varphi}$ and $U^{\varphi}=\TT^{\varphi}U^{\varphi}$.
\end{itemize}
\end{lem}

\begin{proof}
{\rm (a)} First, using that $\pz$ is Markovian and (\ref{dpp-1})-(\ref{dpp-3}), we have
\begin{align*}
&\sum_{m=1}^{n+1}\hat{\EE}_{(s,x)}^\pi \bigg[   ((s - \hat{S}_m)^+\wedge \hat{T}_{m+1})\hat{c}(\hat{X}_m,A_m)+\one_{[\hat{S}_m,\hat{S}_{m+1})}(s)\hat{g}(\hat{X}_m,A_m) \bigg]\\
=&\sum_{m=1}^{n+1}\sum_{a\in A(x)}\varphi_0(a|\lz,x)\int_{\RR_+}\int_{\hat{E}} \hat{Q}(\d t,\d y|x,a)\bigg[  \hat{\EE}_{(s,x)}^\pi \bigg[  ((s -\hat{S}_m)^+\wedge \hat{T}_{m+1})\hat{c}(\hat{X}_m,A_m)\\
& \qd  +\one_{[\hat{S}_m,\hat{S}_{m+1})}(s)\hat{g}(\hat{X}_m,A_m) \big|\hat{X}_1=y,\hat{T}_1=t\bigg]\bigg]\\
=&\sum_{a\in A(x)}\varphi_0(a|\lz,x)\int_{\RR_+}\int_{\hat{E}} \hat{Q}(\d t,\d y|x,a)\sum_{m=1}^{n+1}\hat{\EE}_{((s-t),y)}^{^{(1)}\pi}\bigg[  ((s-t) -\hat{S}_{m-1})^+\wedge \hat{T}_{m})\\
& \qd\times\hat{c}(\hat{X}_{m-1},A_{m-1})  +\one_{[\hat{S}_{m-1},\hat{S}_{m})}(s-t)\hat{g}(\hat{X}_{m-1},A_{m-1})\bigg]\\
=&\sum_{a\in A(x)}\varphi_0(a|s,x)\int_{[0,s]}\int_{\hat{E}}U_n^{^{(1)}\pz}(s-t,y) Q(\d t,\d y|x,a),
\end{align*}
where the last equality is due to that if $t>s$, $\one_{[\hat{S}_{m-1},\hat{S}_{m})}(s-t)=0$ and $((s-t) -\hat{S}_{m-1})^+=0$. Then, we have
\begin{align*}
U^\pi_{n+1}(\lz,x) 
=&\sum_{a\in A(x)}\varphi_0(a|s,x)\bigg[\hat{c}(x,a)\int_{\RR_+} (t\wedge s)\hat{Q}(\d t,\hat{E}|x,a)\\
&+\hat{g}(x,a)\int_{\RR_+}\one_{[0,t)}(s)\hat{Q}(\d t,\hat{E}|x,a)
+\int_{[0,s]}\int_{\hat{E}} U^{^{(1)}\pi}_{n}(s-t,y)\hat{Q}(\d t, \d y|x,a)\bigg]\\
=&\sum_{a\in A(x)}\varphi_0(a|s,x)\bigg[ \hat{c}(x,a)\int_{0}^{\lz} (1-\hat{Q}(t,\hat{E}|x,a)\d t+ \hat{g}(x,a) (1-\hat{Q}(s,\hat{E}|x,a))\\
&+\int_{[0,s]} \int_{\hat{E}} U^{^{(1)}\pi}_{n}(s-t,y) \hat{Q}(\d t, \d y|x,a) \bigg]\\
=&\TT^{\varphi_0}U^{^{(1)}\pi}_{n}(s,x).
\end{align*}
Further, noting that $A(x)$ is finite, under Assumption \ref{ass1}, the monotone convergence theorem implies $U^\pi=\TT^{\varphi_0}U^{^{(1)}\pi}$.

{\rm (b)} It immediately follows from part {\rm (a)}. 
\end{proof}

At the end of this section, we state our main results about SMDPs with the terminal cost function. In detail, we provide an iterative algorithm for computing the value function $U^*$, and give the optimality equation and the existence of optimal policies.

\begin{thm}\lb{smdp-thm1}
Suppose that Assumption \ref{ass1} holds. Then the following statements hold.
\begin{itemize}
\item [\rm (a)] (Value iteration) For every $n\geq 0$, let $U^*_{0}=:0$ and $U^*_{n+1}:=\TT U^*_{n}$. Then, $U^*=\lim_{n \to \infty}U^*_n\in\MM$.
\item [\rm (b)] (Optimality equation) $U^*$ is the minimum solution in $\MM$ to the optimality equation $U^*=\TT U^*$, that is, if $u\in \MM$ satisfies that $u=\TT u$, then $u\geq U^*$.
\item [\rm (c)] (Optimal policy) There exists an $f^*\in \FF$ such that $U^*=\TT^{f^*} U^*$, and such a policy $f^*\in \FF$ is $T$-horizon optimal.
\end{itemize}	
\end{thm}
\prf
{\rm (a)} Since $\hat{c}(x,a)$ and $\hat{g}(x,a)$ are nonnegative and $\TT$ is a nondecreasing map from $\MM$ to $\MM$, we obtain $U^*_{n+1}(\lz,x)\geq U_n^*(\lz,x)$ and $U^*_n\in \MM$ for all $n\geq 0$ by $U_{0}^*=0$. Therefore, $u^*:=\lim_{n \to \infty}U^*_n\in \MM$. To prove part ({\rm a}), it remains to establish that $u^*=U^*$. We show $u^*\leq U^*$ and $U^*\geq u^*$, respectively.

To show $u^*\leq U^*$, we prove that
\be\lb{s2-thm1-1}
U^*_{n+1}\leq U_n^\pz, \qd \forall n\geq -1 ,\pz\in \ppz_{RM},
\de
and do this by induction. It is obviously true for $n\geq -1$. Suppose that $U^*_{n+1}\leq U_n^\pz$ for some $n\geq -1$ and any $\pz\in \ppz_{RM}$. Then, fixed any $\pz=\{\varphi_n,n\geq 0\}\in \ppz_{RM}$, by Lemma \ref{smdp-lemma1} part (a), it holds  that
$$
U_{n+1}^\pz=\TT^{\varphi_0}U^{^{(1)}\pz}_n \geq \TT^{\varphi_0}U^*_{n+1} \geq \TT U^*_{n+1}=U^*_{n+2},
$$
where ${^{(1)}\pz}=\{\varphi_n,n\geq 1\} \in \ppz_{RM}$ and the second and third inequalities follow from inductive hypothesis and the definitions of $\TT$ and $\TT^{\varphi_0}$, respectively. Hence (\ref{s2-thm1-1}) holds, and thus $u^*\leq U^*$.

We now show that $u^*\geq U^*$. For every fixed $(\lz,x)\in [0,T]\times \hat{E}$, since $A(x)$ is finite, there exists an $a^*_n(\lz,x)$ satisfying that $\TT^{a^*_n(\lz,x)}U^*_{n}(\lz,x)=\TT U^*_{n}(\lz,x)=U^*_{n+1}(\lz,x)$. Using that $A(x)$ is finite again, there is a subsequence $\{n_k\}$ of $\{n\}$ and $a^*(\lz,x)\in A(x)$ such that $a_{n_k}^*(\lz,x)=a^*(\lz,x)$ for all $n_k$. Hence, $U^*_{n_k+1}(\lz,x)=\TT^{a^*(\lz,x)}U^*_{n_k}(\lz,x)$. Letting $n_k \to \infty$, it holds that $u^*(\lz,x)=\TT^{a^*(\lz,x)}u^*(\lz,x)$ by monotone convergence theorem, which implies that $u^*(\lz,x)\geq \TT u^*(\lz,x)$. By the arbitrariness of $(\lz,x)$, we have $u^*\geq \TT u^*$. On the other hand, the finiteness of $A(x)$ and measurable selection theorem (see \cite[Proposition D.5]{HL96}) ensure that there is an $f^*\in \FF$ such that
$$
\TT^{f^*}u^*=\TT u^* \leq u^*.
$$
Moreover, Since $u^*\geq 0=U^{f^*}_{-1}$, by induction it holds that $u^*\geq \TT^{{f^*}}U_{n}^{f^*}= U_{n+1}^{f^*}$ for all $n\geq -1$, which implies
\be\lb{s3-thm-f}
u^*\geq \lim_{n\to \infty}U_n^{f^*}=U^{f^*} \geq U^* \geq u^*
\de
Therefore, $u^*=U^*$.

{\rm (b)} For every $\pz=\{\varphi_n,n\geq 0\}\in \ppz_{RM}$, by Lemma \ref{smdp-lemma1} (a), it holds that
$$
U^{\pz}=\TT^{\varphi_0}U^{^{(1)}\pz}\geq \TT^{\varphi_0}U^* \geq \TT U^*.
$$
Then, the arbitrariness of $\pz$ implies that $U^* \geq \TT U^*$. On the other hand,
$$
U^*_{n+1}(\lz,x)=\TT U^*_{n}(\lz,x) \leq \TT^a U^*_{n}(\lz,x), \qd \forall (\lz,x)\in [0,T]\times \hat{E}, a\in A(x), n\geq 0.
$$
Hence, by the monotone convergence theorem, we obtain that $U^*(\lz,x)\leq \TT^a U^*(\lz,x)$, and so $U^*(\lz,x)\leq  \TT U^*(\lz,x) $. Therefore, $U^*=\TT U^*$.

Let $u\in \MM$ be an arbitrary solution to the equation $u=\TT u$. Since $u\geq 0=U^*_{0}$ and $u=\TT^n u$, it follows from part (a) that
$$u=\lim_{n\to \infty}\TT^{n+1} u\geq \lim_{n\to \infty}\TT^{n+1}  U^*_{0}=\lim_{n \to \infty } U^*_n=U^*,$$
where $T^1 v=Tv$ and $T^n v=T(T^{n-1}v)$ for all $n\geq 1$ and $v\in \MM$. This means that $U^*$ is the minimum solution in $\MM$ to the optimality equation.

{\rm (c)} By the proof of part (a), there is an $f^*\in \FF$ such that $\TT U^*=\TT^{f^*}U^*$. Therefore, part (b) gives $U^*=\TT^{f^*}U^*$. Hence, $f^*$ is $T$-horizon optimal by part (a) and (\ref{s3-thm-f}).\deprf

\begin{rem}
In particular, if $\hat{E}$ is denumerable and $\hat{g}=0$, the above results are same as \cite[Theorem 3.1 and Theorem 3.2]{HG11}.
\end{rem}

\section{Existence and computation of optimal stopping times}
In this section, we introduce the equivalent SMDPs corresponding to the original optimal stopping problem of SMPs in section 2. And then, we show that for every stopping time $\tz$ and $s\in [0,T]$, there is a policy $\pi_\tz$ such that the $s$-horizon expected cost of $\tz$ is equal to $s$-horizon expected cost of the policy $\pi_\tz$. Hence, we can analyze the value function $V^*$ and the optimal stopping time $\tz^*$ of SMPs through the conclusions of SMDPs given in section 3. Note that the regular condition (Assumption \ref{s2-ass1}) is needed.

Intuitively, in the SMPs, ``continue" or ``stop" can be considered as a special action in the corresponding SMDPs. This intuition gives us an idea to construct the SMDPs. The details about the constructions of SMDPs are given as follows. Hence, the model of the corresponding SMDPs is
\be\lb{m-smdp}
\lt\{\hat{E}, A, (A (x) \subset A,x\in \hat{E}), \hat{Q}_T (\cdot, \cdot |x, a), \hat{c}(x, a), \hat{g}(x,a) \rt\}.
\de
where the state space $\hat{E} := E \cup \{ \Delta \}$ includes the state space $E$ of SMPs and a virtual state $\Delta$. $A(x)$, denoting the set of admissible actions at state $x\in  \hat{E}$, is defined as
$$
A (x):=
\left\{
\begin{array}{ll}
\{0, 1 \}, & x\in E; \\
\{ 1 \}, & x= \Delta,
\end{array}
\right.
$$
where the action $0$ means continuation and $1$ means stop. The action space $A =\{0,1\}$ is finite. Then the set of admissible state-action pairs $K=(E\times A)\cup \{(\Delta,1)\}$ is a Borel subset of $\hat{E}\times A$. For each $t\geq 0$ and $B\in \Bo(\hat{E})$, the semi-Markov kernel $\hat{Q}_T (t, B |x, a)$ is given by
\be\lb{Q}
\hat{Q}_T (t, B | x, a) :=
\left\{
\begin{array}{ll}
Q(t, B\setminus \{\Delta\} | x), & x \in E, a = 0; \\
	\one_{ \lt[T+1, +\infty \rt) } (t)\delta_{\Delta}(B), & x \in \hat{E},  a=1,
\end{array}
\right.
\de
where $	Q(t, B | x)$ is the kernel of SMPs given in (\ref{smp}). Finally, the cost rate function and terminal cost function of SMDP are defined as
\be\lb{c}
\hat{c}(x, a) :=
\left\{
\begin{array}{ll}
	c (x), & x \in E, a = 0; \\
	0, & \text{otherwise},
\end{array}
\right.
\de
\be\lb{g}
\hat{g}(x, a) :=
\left\{
\begin{array}{ll}
	g (x), & x \in E, a = 1; \\
	0, & \text{otherwise},
\end{array}
\right.
\de
where $c$ and $g$ are the cost rate function and terminal cost function of SMPs, respectively. Since $c$ and $g$ are measurable on $E$, $\hat{c}(x,a)$ and $\hat{g}(x,a)$ are measurable on $K$.  Firstly, we give a lemma to show that the above model satisfies Assumption \ref{ass1}.

\begin{lem}\lb{s4-lem1}
Suppose that Assumption \ref{s2-ass1} holds. For the SMDPs as in (\ref{m-smdp}), Assumption \ref{ass1} is fulfilled.
\end{lem}

\begin{proof}
Under Assumption \ref{s2-ass1}, there exist $\dz>0$ and $\epsilon > 0$ such that (\ref{ass}) holds. Let $\hat{\delta}= \min\{\dz,\frac{1}{2}\}$, and thus
\begin{equation*}%\lb{pf-u}
\hat{Q}_T (\hat{\delta}, \hat{E} | x, a) =
\left\{
\begin{array}{ll}
Q(\hat{\delta},E | x) \leq 	Q(\delta, E | x) \leq  1-\epsilon, & x \in E,  a = 0; \\
\one_{ \lt[T+1, +\infty \rt) } (\hat{\delta}) =0 \leq 1-\epsilon, & i \in \hat{E},  a=1.
\end{array}
\right.
\end{equation*}
Hence, Assumption \ref{ass1} is fulfilled by \cite[Proposition 2.1]{HG11}.
\end{proof}

Next step, we will show the relationship between the stopping times $\tz \in \ggz$ of SMPs and the policies $\pz \in \ppz_{DH}$ of SMDPs as in (\ref{m-smdp}). To do so, for all $n\geq 0$ and history of SMPs up to the $n$-th jump epoch $h_n=(x_0,t_1,\ldots,x_{n-1},t_n,x_n)\in {H}_n$, let
\be\lb{Mn}
 h_n^0= \lt(x_0, 0, t_1, x_1, \ldots,0, t_n, x_n \rt) \in \hat{H}_n.
\de
The action $0$ (means continuation) added to the equation (\ref{Mn}) indicates that the system has been running incessantly before the $n$-th jump epoch. Obviously, by the definition of $h_n^0$, it holds that
\be\lb{HN0}
C^0:=\{h_n^0\,|\,h_n\in C\}\in \Bo(\hat{H}_n), \qd \forall C\in \Bo({H}_n).
\de
In particular, ${H}_n^0=E\times (\{0\}\times \RR_+\times E)^n \in \Bo(\hat{H}_n)$.

\begin{defn}\lb{s3-t-p}
Given any deterministic policy $\pi=\{d_n,n\geq 0\}\in \ppz_{DH}$ defined in Definition \ref{pclass} (c) and $s \in \RR$, define
$$
\tz_\pz^s (\oz) := \inf \bigg\{ n \in \mathbb{N}\,\bigg|\, d_n (\lz-S_n(\oz), Y_n^0 (\oz)) = 1 \bigg\},\qd \oz=(x_0,t_1,\ldots,x_n,t_{n+1},\ldots) \in\ooz,
$$
where $\inf \{ \emptyset \} := +\infty$ and $ Y_n^0:=(X_0,0,T_1,X_1,\ldots,0,T_n,X_n)$. Then $\tau_\pi^s$ is called the stopping time induced by the policy $\pi$ and $s$.
\end{defn}

\begin{lem}\lb{s3-lem2}
For each deterministic policy $\pi=\{d_n,n\geq 0\}$ and $s \in \RR$, the induced stopping time $\tau_\pi^s$ is a stopping time.	
\end{lem}

\begin{proof}
Note that for each $n \geq 0$, the random variables $Y_n^0$ and $S_n$, and the function $d_n$ are measurable in their corresponding spaces. Hence, we have
$$
\{ \tz_\pz^s = n \} = \big( \cap_{k=0}^{n-1} \{ d_k ((s -S_k), Y_k^0) = 0 \} \big) \cap \big\{ d_n ((s- S_{n}) ,Y_n^0) = 1 \big\} \in \sigma(Y_n)=\fq_n,
$$
which implies that $\tau_\pi^s$ is a stopping time.
\end{proof}

We introduce a subclass $\ppz^0_{DH}$ of $\ppz_{DH}$ by
$$
\ppz^0_{DH}:=\{\pz=\{d_n,n\geq 0\}\in \ppz_{DH}\,|\,d_n(0,\hat{h}_n)=0,\forall n\geq 0,\hat{h}_n\in {H}^0_n\}.
$$
We now give a key theorem which establishes the relationship between the $s$-horizon expected cost of polices in $\ppz^0_{DH}$ and that of the stopping times induced by the polices.

\begin{thm}\lb{s3-thm1}
Suppose that Assumption \ref{s2-ass1} holds. For any $\pi=\{d_n,n\geq 0\}\in \ppz_{DH}^0$, it holds that
$$
U^{\pi}(s,x)=V^{\tau_\pi^s}(s,x), \qd \forall x \in E, \lz\in [0,T],
$$
where $\tau_{\pi}^s $ is the stopping time induced by $\pz$ and $s$, and $V^{\tau_\pi^s}(s,x)$ is the $s$-horizon expected cost of $\tau_{\pi}^s$.
\end{thm}

\begin{proof}
For each $\hat{\oz}= (x_0, a_0, t_1,\ldots, x_n, a_n,t_{n+1} ,\ldots) \in \hat{\ooz}$, recall that $\hat{S}_n(\oz)=\sum_{k=1}^n t_k$ and
$$\hat{Y}_n (\hat{\oz})= (x_0, a_0,t_1,x_1 \ldots,a_{n-1}, t_n, x_n).$$
We define $C_n(n\geq 0)$ and $C$, the subsets of $\hat{\ooz}$, as
\begin{align*}
C_n & := \lt\{ \oz \in \hat{\ooz} : \inf \lt\{ k \in \NN: d_k (\lz-\hat{S}_{k}(\oz),\hat{Y}_k(\oz)) = 1 \rt\} =n \rt\}, \qd n \geq 0;\\
C & := \lt\{  \oz \in \hat{\ooz} : d_k (\lz-\hat{S}_{k}(\oz),\hat{Y}_k(\oz)) = 0, \forall k \geq 0\rt\}.
\end{align*}
It is easy to know that $\{C, C_n, n\geq 0 \}$ is a partition of $\hat{\ooz}$ and
$$
\one_{C_n}=\prod_{k=0}^{n-1}\one_{\{0\}}(d_k(\lz-\hat{S}_k,\hat{Y}_k)) \times \one_{\{1\}}(d_n(\lz-\hat{S}_n,\hat{Y}_n)).
$$
Hence, the monotone convergence theorem implies that
\begin{align}\lb{pf-s3-thm1-0}
U^{\pi}(s,x)=&\sum_{n=0}^\infty\hat{\EE}^\pi_{(s,x)}\lt[ \one_{C_n} \lt(\int_0^\lz \hat{c}(\hat{X}(t),A(t)) \d t+ \hat{g}(\hat{X}(s),A(s))\rt) \rt]\notag\\
&+\hat{\EE}^\pi_{(s,x)}\lt[ \one_{C} \lt(\int_0^\lz \hat{c}(\hat{X}(t),A(t))\d t+ \hat{g}(\hat{X}(s),A(s))\rt) \rt].
\end{align}
Noting that $\pz$ is a deterministic, the definition of $C_n$ and (\ref{dpp-2}) give
\be\lb{pf-s3-thm1-a}
\hat{\PP}_{(\lz,x)}^{\pz} \lt[ A_k=1|C_n \rt] = \hat{\PP}_{(\lz,x)}^{\pz} \lt[ d_k(\lz-\hat{S}_k,\hat{Y}_k)=1\big|C_n \rt]=
\left\{
\begin{array}{ll}
	0, & k < n; \\
	1, & k  = n,
\end{array}
\right.
\de
which, together with $\lim_{t\to \infty}\hat{Q}_T(t,\Delta|x,1)=1$ and $A(\Delta)=\{1\}$, implies that $\hat{\PP}_{(\lz,x)}^{\pz} (\hat{X}_m=\Delta|C_n)=1$ for all $m>n$. Thus, by Lemma \ref{s4-lem1}, (\ref{vn}), (\ref{c}) and (\ref{g}) give
\begin{align}\lb{pf-s3-thm1-1}
&\hat{\EE}^\pi_{(\lz,x)}\lt[ \one_{C_n} \lt(\int_0^\lz \hat{c}(\hat{X}(t),A(t))\d t+ \hat{g}(\hat{X}(s),A(s))\rt) \rt]\notag\\
=&\sum_{m=0}^{n-1} \hat{\EE}^\pi_{(\lz,x)} \bigg[ \hat{c}(\hat{X}_m,A_m) ((\lz -\hat{S}_m)^+\wedge \hat{T}_{m+1}) \one_{C_n} \bigg]+ \hat{\EE}^\pi_{(\lz,x)} \bigg[ \one_{C_n} \one_{[\hat{S}_n,\hat{S}_{n+1})}(\lz)\hat{g}(\hat{X}_n,1) \bigg].
\end{align}
And then, for each $m<n$, using
\be\lb{relation}
\one_{\{\tz^\lz_{\pz}=n\}}=\prod_{k=0}^{n-1} \one_{\{0\}}(d_k(\lz-S_k,Y_k^0)) \cdot \one_{\{1\}}(d_n(\lz-S_n,Y_n^0))
\de
and $\hat{Q}_T(\cdot,B|x,0)=Q(\cdot,B\setminus \{\Delta\}|x)$ for all $x\in E$ and $B\in \bq(\hat{E})$, it holds that
\begin{align}\lb{pf-s3-thm1-2}
&\hat{\EE}^\pi_{(\lz,x)} \bigg[ \hat{c}(\hat{X}_m,A_m)  ((\lz -\hat{S}_m)^+\wedge \hat{T}_{m+1}) \one_{C_n}\bigg]\notag\\
=&\int_{\hat{E}}\delta_{x}(\d x_0)\int_{\hat{E}}\int_{\RR_+}\hat{Q}_{T}(\d t_1,\d x_1|x_0,d_0(\lz,\hat{h}_0))\int_{\hat{E}}\int_{\RR_+}\hat{Q}_{T}(\d t_2,\d x_2|x_1,d_0(\lz-s_1,\hat{h}_1))\notag\\
& \cdots \int_{\hat{E}}\int_{\RR_+}\hat{Q}_{T}(\d t_n,\d x_n|x_{n-1},d_{n-1}(\lz-s_{n-1},\hat{h}_{n-1}))\hat{c}(x_m,d_m(\lz-s_m,\hat{h}_m))\notag\\
&((\lz-s_m)^+\wedge t_{m+1})\prod_{k=0}^{n-1}\one_{\{0\}}(d_k(\lz-s_k,\hat{h}_k)) \cdot \one_{\{1\}}(d_n(\lz-s_n,\hat{h}_n))\notag\\
=&\int_{E}\delta_{x}(\d x_0)\int_{E}\int_{\RR_+}Q(\d t_1,\d x_1|x_0) \cdots  \int_{E}\int_{\RR_+}Q(\d t_n,\d x_n|x_{n-1})c(x_m)\notag\\
&((\lz-s_m)^+\wedge t_{m+1})\prod_{k=0}^{n-1}\one_{\{0\}}(d_k(\lz-s_k,h_k^0)) \cdot \one_{\{1\}}(d_n(\lz-s_n,h_n^0))\notag\\
=&\EE_x\lt[((\lz-S_m)^+\wedge T_{m+1})c(X_m)\one_{\{\tz^\lz_{\pz}=n\}} \rt],
\end{align}
where $s_k=\sum_{i=1}^k t_i$, $\hat{h}_0=x_0$, $\hat{h}_{k+1}=(\hat{h}_k,d_{k}(\lz-s_k,\hat{h}_k),t_{k+1},x_{k+1})$, $h_k=(x_0,t_1,x_1,\ldots,t_{k},x_k)$, and $h_k^0$ defined in (\ref{Mn}).
For each $\hat{\oz}\in C_n$,  $\pz\in \ppz_{HD}^0$ gives $\hat{Y}_n(\hat{\oz})\in H_n^0$, and then $d_n(0,\hat{Y}_n(\hat{\oz}))=0$.
Further, by $d_n(\lz-\hat{S}_n(\hat{\oz}),\hat{Y}_n(\hat{\oz}))=1$ for all $\hat{\oz}\in C_n$,  $ C_n \cap \{\lz=\hat{S}_n\}=\emptyset$. Therefore, we have $\one_{C_n}\one_{[\hat{S}_n,\hat{S}_{n+1})}(\lz)=\one_{C_n}\one_{(\hat{S}_n,\hat{S}_{n+1})}(\lz)$. Thus, (\ref{Q})   and (\ref{g}) show
\begin{align}\lb{pf-s3-thm1-3}
&\hat{\EE}_{(\lz,x)}^\pi \bigg[ \one_{C_n} \one_{[\hat{S}_n,\hat{S}_{n+1})}(\lz)\hat{g}(\hat{X}_n,1) \bigg]\notag\\
=&\int_{\hat{E}}\delta_{x}(\d x_0)\int_{\hat{E}}\int_{\RR_+}\hat{Q}_{T}(\d t_1,\d x_1|x_0,d_0(\lz,\hat{h}_0))\int_{\hat{E}}\int_{\RR_+}\hat{Q}_{T}(\d t_2,\d x_2|x_1,d_0(\lz-\lz_1,\hat{h}_1))\notag\\
& \cdots  \int_{\hat{E}}\int_{\RR_+}\hat{Q}_{T}(\d t_{n+1},\d x_{n+1}|x_{n},d_{n}(\lz-s_{n},\hat{h}_{n})) \hat{g}(x_n,1)\one_{(s_n,s_n+t_{n+1})}(\lz)\notag\\
&\times\prod_{k=0}^{n-1}\one_{\{0\}}(d_k(\lz-s_k,\hat{h}_k)) \cdot \one_{\{1\}}(d_n(\lz-s_n,\hat{h}_n))\notag\\
=&\int_{E}\delta_{x}(\d x_0)\int_{E}\int_{\RR_+}Q(\d t_1,\d x_1|x_0) \cdots \int_{E}\int_{\RR_+}Q(\d t_{n},\d x_{n}|x_{n-1}) g(x_n)\one_{(s_n,s_n+T+1)}(\lz)\notag\\
&\times \prod_{k=0}^{n-1}\one_{\{0\}}(d_k(\lz-s_k,h_k^0)) \cdot \one_{\{1\}}(d_n(\lz-s_n,h_n^0)) \qd \text{(by $\hat{Q}_T(t,\hat{E}|x_n,1)=\one_{[T+1,\infty)}(t))$}\notag\\
=&\EE_x \lt[\one_{(S_n,+\infty)}(\lz)g(X_n)\one_{\{\tz^\lz_{\pz}=n\}} \rt].
\end{align}
Moreover, by the definition of $R^{s}_{\tau_{\pz}^s}$ given in (\ref{r-t}), (\ref{pf-s3-thm1-1}), (\ref{pf-s3-thm1-2}), and (\ref{pf-s3-thm1-3}), we have
\begin{align}\lb{pf-s3-thm1-4}
&\EE_x\lt [ \one_{\{\tz^\lz_{\pz}=n\}}R^\lz_{\tz^\lz_{\pz}}\rt] \notag\\
=&\EE_x\lt [ \one_{\{S_n<\lz\}}\one_{\{\tz^\lz_{\pz}=n\}}\lt(\int_0^{S_n}c(X(t))\d t+g(X(S_n))\rt)\rt]+\EE_x\lt [ \one_{\{S_n\geq \lz\}}\one_{\{\tz^\lz_{\pz}=n\}}\int_0^{\lz}c(X(t))\d t\rt]\notag\\
=& \EE_x\lt [\one_{\{S_n<\lz\}}\one_{\{\tz^\lz_{\pz}=n\}} g(X_n)\rt]+\EE_x\lt [\one_{\{S_n<\lz\}}\one_{\{\tz^\lz_{\pz}=n\}} \sum_{m=0}^{n-1}c(X_m)((\lz-S_m)^+\wedge T_{m+1} ) \rt]\notag\\
&+\EE_x\lt [\one_{\{S_n\geq \lz\}}\one_{\{\tz^\lz_{\pz}=n\}} \sum_{m=0}^{n-1}c(X_m) ((\lz-S_m)^+\wedge T_{m+1}) \rt]\notag\\
=&\hat{\EE}^\pi_{(\lz,x)}\lt[ \one_{C_n} \lt(\int_0^{\lz} \hat{c}(\hat{X}(t),A(t))\d t+\hat{g}(\hat{X}(\lz),A(\lz))\rt) \rt].
\end{align}
Next, we calculate the second item of (\ref{pf-s3-thm1-0}). Using that $\{C, C_n, n\geq 0 \}$ is a partition of $\hat{\ooz}$ again, for all $k\geq 0$ we obtain that
$$\one_C=\one_C \times \prod_{m=0}^{k}(1-\one_{C_m}) = (1-\sum_{n=0}^{+\infty }\one_{C_n}) \times  \prod_{m=0}^{k}(1-\one_{C_m})
=\prod_{m=0}^{k}(1-\one_{C_m})- \sum_{n=k+1}^{+\infty } \one_{C_n},$$
which, combining with (\ref{vn}), implies that
\begin{align*}
&\hat{\EE}^\pi_{(\lz,x)}\lt[ \one_{C} \lt(\int_0^{\lz} \hat{c}(\hat{X}(t),A(t))\d t+\hat{g}(\hat{X}(\lz),A(\lz))\rt) \rt]\\
=&\sum_{k=0}^\infty \bigg\{\hat{\EE}^\pi_{(\lz,x)} \bigg[ \prod_{m=0}^{k}(1-\one_{C_m}) \bigg( ((\lz -\hat{S}_k)^+\wedge \hat{T}_{k+1})\hat{c}(\hat{X}_k,A_k)+\one_{[\hat{S}_k,\hat{S}_{k+1})}(\lz)\hat{g}(\hat{X}_k,A_k)\bigg) \bigg] \\
&\qqd- \sum_{n=k+1}^{+\infty }\hat{\EE}^\pi_{(\lz,x)} \bigg[\one_{C_n} \bigg(((\lz-\hat{S}_k)^+\wedge \hat{T}_{k+1})\hat{c}(\hat{X}_k,A_k)+\one_{[\hat{S}_k,\hat{S}_{k+1})}(\lz)\hat{g}(\hat{X}_k,A_k) \bigg)\bigg] \bigg\}.
\end{align*}
Firstly, using that $\hat{g}(x,0)=0$ for all$x\in E$, (\ref{pf-s3-thm1-a}) and (\ref{pf-s3-thm1-2}), we obtain that for each $n>k$,
\begin{align*}
&\hat{\EE}^\pi_{(\lz,x)} \bigg[\one_{C_n} \bigg(((\lz-S_k)^+\wedge \hat{T}_{k+1})\hat{c}(\hat{X}_k,A_k)+\one_{[\hat{S}_k,\hat{S}_{k+1})}(\lz)\hat{g}(\hat{X}_k,A_k) \bigg)\bigg]\\
=&\EE_x\lt[((\lz-S_m)^+\wedge T_{m+1})c(X_m)\one_{\{\tz^\lz_{\pz}=n\}} \rt].
\end{align*}
And for any $k\geq 0$, by the same methods of (\ref{pf-s3-thm1-2})
\begin{align*}
&\hat{\EE}^\pi_{(\lz,x)} \bigg[ \prod_{m=0}^{k}(1-\one_{C_m}) \bigg( ((\lz -\hat{S}_k)^+\wedge \hat{T}_{k+1})\hat{c}(\hat{X}_k,A_k)+\one_{[\hat{S}_k,\hat{S}_{k+1})}(\lz)\hat{g}(\hat{X}_k,A_k)\bigg) \bigg] \\
=&\EE_x\lt[ \prod_{m=0}^k (1- \one_{\{\tz^\lz_{\pz}=m\}})((\lz-S_k)^+\wedge T_{k+1})c(X_k) \rt].
\end{align*}
Therefore, we obtain that
\begin{align*}
\EE_x\lt[ \one_{\{\tz^\lz_{\pz}=+\infty\}}R^\lz_{\tau^\lz_{\pz}}\rt]
=\hat{\EE}^\pi_{(\lz,x)}\lt[ \one_{C} \lt(\int_0^{\lz} \hat{c}(\hat{X}(t),A(t))\d t+ \hat{g}(X(\lz),A(\lz))\rt) \rt],
\end{align*}
which, together with (\ref{pf-s3-thm1-0}) and (\ref{pf-s3-thm1-4}), shows that
$$
U^{\pz}(\lz,x)=\sum_{n=0}^{\infty}\EE_x\lt[ \one_{\{\tz^\lz_{\pz}=n\}}R^\lz_{\tau^\lz_{\pz}}\rt]+\EE_x\lt[ \one_{\{\tz^\lz_{\pz}=+\infty\}}R^\lz_{\tau^\lz_{\pz}}\rt]=V^{\tz^\lz_{\pz}}(\lz,x), \forall x\in E ,\lz\in [0,T].
$$
The proof of Theorem \ref{s3-thm1} is completed.
\end{proof}

For any $\lz\in [0,T] $, the Definition \ref{s3-t-p}, Lemma \ref{s3-lem2} and Theorem \ref{s3-thm1} say that for each policy $\pi\in \Pi_{DH}^0$, we can construct a stopping time $\tau^s_\pi$ such that their $s$-horizon expected costs are equal. On the other hand, for each stopping time we also can construct a policy which satisfies this condition, see Definition \ref{s3-p-t}, Lemma \ref{s3-lem1} and Theorem \ref{tpt-prop1}.

\begin{defn}\lb{s3-p-t}
Given any stopping time $\tz \in \ggz$ and $n \geq 0$, let
$$
B_n^\tau:=\big\{Y_n(\oz):\oz=(x_0,t_1,x_1,\ldots,t_k,x_k,\ldots)\in\{\tau=n\} \big\},
$$
where $Y_n(\oz)=(x_0, t_1, x_1, \ldots, t_n, x_n)$.	For each $\hat{h}_n = (x_0, a_0, t_1, x_1,\ldots, a_{n-1},  t_n, x_n) \in \hat{{H}_n}$, $s \in \RR$,  define
$$
d_n^{\tz} (\lz,\hat{h}_n) :=
\left\{
\begin{array}{ll}
\one_{B_n^{\tz}} (x_0, t_1, x_1, \ldots, t_n, x_n)\one_{(0,\infty)}(\lz), & \hat{h}_n \in {H}_n^0; \\
1, & \hat{h}_n \in \hat{H}_n \setminus{H}_n^0,
\end{array}
\right.
$$
where ${H}_n^0=E\times (\{0\}\times \RR_+\times E)^n$. $\pz_\tz := \{ d_n^\tz, n \geq 0 \}$ is called the policy induced by $\tz$.
\end{defn}

\begin{lem}\lb{s3-lem1}
For each stopping time $\tz \in \ggz$, $\pi_\tau$ is in $\ppz^0_{DH}$ on the corresponding SMDPs.
\end{lem}

\begin{proof}
By Definition \ref{s3-p-t}, it can verify easily that $d_n^\tau (\lz,\hat{h}_n)\in A(x_n)$. Then, we just need to consider the measurability. Noting that $\{\tau=n\}\in \fq_n=\sigma (Y_n)$, we obtain $B_n^\tau \in \Bo ({H}_n)$. Thus, we have
$$
\big\{ (\lz,\hat{h}_n) \in \RR\times\hat{H}_n | d_n^\tau (\lz,\hat{h}_n) = 0 \big\} = \bigg( (0,+\infty)\times ( {H}_n \setminus B_n^\tau)^0 \bigg)\cup
\bigg( (-\infty,0] \times{H}_n^0\bigg)\in \Bo (\RR \times \hat{H}_n),
$$
where $( {H}_n \setminus B_n^\tau)^0$ is defined in (\ref{HN0}). %And then
Hence, $\pz_\tz := \{ d_n^\tz, n \geq 0 \}$ is a deterministic policy of the corresponding SMDPs. Furthermore, for each $\hat{h}_n\in {H}_n^0$, it holds that $f_n^\tau(0,\hat{h}_n)=0$, which implies that $\pz_\tz\in \ppz_{DH}^0$.
\end{proof}

\begin{thm}\lb{tpt-prop1}
Suppose that Assumption \ref{s2-ass1} holds. For each stopping time $\tz\in \Gamma$, let $\pi_\tau := \{ d_n^\tau, n\geq 0\}$ be the policy induced by $\tz$. Then,
\be\lb{tpt-1}
V^{\tz}(\lz,x)=U^{\pz_\tz}(\lz,x)=V^{\tz_{\pz_{\tz}}^\lz}(\lz,x) \qd \forall  (\lz,x)\in [0,T] \times E,
\de
where $\tau_{\pi_\tau}^\lz$ is the stopping time induced by $\pi_\tau$ and $\lz$.
\end{thm}

\begin{proof}
To prove (\ref{tpt-1}), by the Theorem \ref{s3-thm1}, it suffices to show that
$
V^{\tz}(\lz,x)=V^{\tz_{\pz_{\tz}}^\lz}(\lz,x).
$
By the (\ref{relation}) and the definition of $d_n^\tz$, we have
\begin{align}\lb{pf-s3-prop1-2}
\one_{\{\tz^\lz_{\pz_\tau}=n\}}=\prod_{k=0}^{n-1} \one_{\{0\}}(d_k^\tau(\lz-S_k,Y_k^0)) \cdot \one_{\{1\}}(d_n^\tau(\lz-S_n,Y_n^0))=\one_{[0,\lz)}(S_n)\one_{\{\tz=n\}},
\end{align}
where $Y_k^0=(X_0,0,T_1,X_1,\ldots,0,T_k,X_k)$. Moreover, using the definition of $\tz^\lz_{\pz_\tau}$ again,
\begin{align}\lb{pf-s3-prop1-3}
\one_{\{\tz^\lz_{\pz_\tau}=\infty\}}
=&\prod_{n=0}^\infty \lt(1-\one_{B_n^{\tz}} (Y_n)\one_{(0,\infty)}(\lz-S_n)\rt)\notag\\
=&\one_{\{\tz=\infty\}}\prod_{n=0}^\infty \bigg(1-\one_{\{\tz=n\}}\one_{(0,\infty)}(\lz-S_n)\bigg)+\one_{\{\tz\neq\infty\}}\prod_{n=0}^\infty \bigg(1-\one_{\{\tz=n\}}\one_{(0,\infty)}(\lz-S_n)\bigg)\notag\\
=&\one_{\{\tz=\infty\}}+\sum_{k=0}^{\infty}\one_{\{\tz=k\}}\one_{[\lz,\infty)}(S_k).
\end{align}
According to the definition of $R_\tau^s$ given in (\ref{r-t}), (\ref{pf-s3-prop1-2}) and (\ref{pf-s3-prop1-3}), we have that
\begin{align*}
V^{\tz_{\pz_{\tz}}^\lz}(\lz,x)
=&\sum_{n=0}^{\infty}\EE_x\lt [\one_{\{\tz^\lz_{\pz_\tau}=n\}}R_{n}^\lz  \rt]+\EE_x\lt [\one_{\{\tz^\lz_{\pz_\tau}=\infty\}}\int_0^\lz c(X(t))\d t  \rt]\\
=&\sum_{n=0}^{\infty}\EE_x\lt [\one_{[0,\lz)}(S_n)\one_{\{\tz=n\}}R_{n}^\lz  \rt]+\EE_x\lt [\lt(\one_{\{\tz=\infty\}}+\sum_{n=0}^{\infty}\one_{\{\tz=n\}}\one_{[\lz,\infty)}(S_n )\rt)\int_0^\lz c(X(t))\d t  \rt]\\
=&\sum_{n=0}^{\infty}\lt(\EE_x\lt [\one_{[0,\lz)}(S_n)\one_{\{\tz=n\}}R_{n}^\lz  \rt]+\EE_x\lt [\one_{\{\tz=n\}}\one_{[\lz,\infty)}(S_n) R_n^\lz  \rt]\rt)+\EE_x\lt [\one_{\{\tz=\infty\}} R_\tz^\lz  \rt]\\
=&V^{\tz}(\lz,x),
\end{align*}
which is the desired result.
\end{proof}

Let $\mathcal{M}$ be the set of Borel measurable functions $u:[0,T]\times E \to [0,\infty]$. Next, we define the operator $\mathbb{G}$ from $\mathcal{M}$ into $\mathcal{M}$ as follows:
$$
\mathbb{G} u (\lz,x) = \min \bigg\{ c(x) \int_0^\lz  (1 - Q(t, E | x) )\d t + \int_{E}\int_{[0,\lz]} u (\lz-t,y) Q( \d t, \d y | x ), g(x) \bigg\},
$$
for each $u\in \mathcal{M}$ and $(\lz,x)\in [0,T]\times E$. Then, we can state our main results of optimal stopping problems, these are an algorithm for computing the value function $V^*$ and a finite optimal stopping time.

\begin{thm}\lb{ost-thm}
{\bf (Value iteration) }Suppose that Assumption \ref{s2-ass1} holds. For any $n\geq -1$, let $V_{0}^*\equiv 0$ and $V_{n+1}^*=\mathbb{G}V_{n}^*$. Then, we have $V^*=\lim_{n \to \infty}V^*_n$ and, moreover, $V^*=\mathbb{G}V^*$.	
\end{thm}

\begin{proof}
We define the policy $f^*$ for the SMDP as in (\ref{m-smdp}) by
\be\lb{op-p}
f^*(\lz,x)=\one_{(0,\infty)\times E}(\lz,x)\one_{\{U^*(\lz,x)=\hat{g}(x,1)\}}(\lz,x)+\one_{\{\Delta\}}(x).
\de
Since $\{U^*(\lz,x)=g(x,1)\}$ is a measurable subset of $\RR \times \hat{E}$ and $f^*(\cdot,\Delta)=1$, $f^*$ is a deterministic stationary policy. Moreover, for any $x\in E$,   we have that $f^*(0,x)=0$, which implies $f^*\in \ppz_{DH}^0$. Next we will show that $f^*$ is $T$-optimal.
\begin{itemize}
\item[\rm (i)] $x=\Delta$: Noting that $A(\Delta)=\{1\}$, by $U^*_{0}\equiv0$ and $\hat{c}(\Delta,1)=\hat{g}(\Delta,1)=0$, it holds that
\begin{align*}
U^*(\lz,\Delta)=\TT U^*(\lz,\Delta)=\TT^{f^*(\lz,\Delta)} U^*(\lz,\Delta)=0,\quad \forall \lz\in [0,T].
\end{align*}
\item [\rm (ii)] $\lz= 0$ and $x\in E$: By $f^*(0,x)=0$, we have $U^*(0,x)=\min\{0,g(x)\}= \TT^{f^*(0,x)} U^*(0,x)$.
\item[\rm (iii)] $(\lz,x)\in ((0,T]\times E)$: Noting that $\TT^1 U^*(\lz,x)=\hat{g}(x,1)$, by $U^*=\TT U^*$, we have
$$
U^*(\lz,x) =
\left\{
\begin{array}{ll}
\TT^1 U^*(\lz,x)= \TT^{f^*(\lz,x)} U^*(\lz,x), & U^*(\lz,x)=\hat{g}(x,1); \\
\TT^0 U^*(\lz,x)= \TT^{f^*(\lz,x)} U^*(\lz,x), & U^*(\lz,x)\neq\hat{g}(x,1).
\end{array}
\right.
$$
\end{itemize}
Hence, $f^*$ is $T$-optimal by Theorem \ref{smdp-thm1} part (c). Therefore, Theorem \ref{s3-thm1} implies that
$$
U^*(\lz,x)= U^{f^*}(\lz,x)= V^{\tz^{\lz}_{f^*}}(\lz,x)\geq V^*(\lz,x), \quad \forall x\in E, \lz\in [0,T].
$$
where $\tz_{f^*}^s$ is the stopping time induced by $f^*$ and $s$. On the other hand, for each $\tau\in \ggz$, Theorem \ref{tpt-prop1} gives
$$
V^{\tz}(\lz,x)=U^{\pz_{\tz}}(\lz,x)\geq U^*(\lz,x)\quad \forall x\in E, \lz\in [0,T].
$$
By the arbitrariness of $\tz$, $V^*(\lz,x)\geq U^*(\lz,x)$. Hence,
\be\lb{relation2}
V^*(\lz,x)=U^*(\lz,x),\quad \forall x\in E, \lz\in [0,T].
\de
Next, we show that for all $n\geq 0$
\be\lb{relation-u-v}
U^*_n(\lz,x)=V^*_n(\lz,x),\quad \forall x\in E, \lz\in [0,T].
\de
Obviously, it holds that for $n=0$. Assume that (\ref{relation-u-v}) holds for some $n$. Then
\begin{align*}
U^*_{n+1}(\lz,x)=&\min\bigg\{ \hat{g}(x,1),\hat{c}(x,0)\int_0^\lz(1- \hat{Q}_T( t,\hat{E}|x,0))\d t\\
&\qqd+ \hat{g}(x,0)(1-\hat{Q}_T(s,\hat{E}|x,0))+\int_{[0,\lz]}\int_{\hat{E}}U^*_{n}(\lz-t,y) \hat{Q}_T(\d t,\d y|x,0)\bigg\}\\
=&\min\bigg\{ g(x),c(x)\int_0^\lz (1-Q( t,E|x))\d t+\int_{[0,\lz]} \int_EV^*_{n}(\lz-t,y) Q(\d t,\d y|x)\bigg\} \\
=&\mathbb{G}V^*_{n}(\lz,x)=V^*_{n+1}(\lz,x),
\end{align*}
where the second equality is due to inductive hypothesis and (\ref{Q})-(\ref{g}). Hence, by Theorem \ref{smdp-thm1} part (a), (\ref{relation2}) and (\ref{relation-u-v}) give that
$V^*(\lz,x)=\lim_{n \to \infty}V^*_n(\lz,x)$. Further, the monotone convergence theorem gives $V^*=\mathbb{G}V^*$.
\end{proof}

\begin{thm}\lb{ost-thm1}
	{\bf (Optimal stopping time) }Suppose that Assumption \ref{s2-ass1} holds. Define a subset of $(0,T]\times E$ by
	\be\lb{st-set}
	D^*:=\{(\lz,x)\in (0,T] \times E: V^*(\lz,x)=g(x)\},
	\de
	and for each $\oz=(x_0,t_1,x_1,\ldots,t_n,x_n,\ldots)\in \ooz$, define
	\be\lb{f-ost}
	\tz^*(\oz):=\inf\bigg\{n\big|(T-\sum_{k=1}^n t_k,x_n)\in D^* \bigg\} \wedge \inf\bigg\{n\big|T\leq\sum_{k=1}^n t_k\bigg\}
	\de
	Then, $\tz^*$ is a $T$-optimal stopping time and satisfies that $\PP_x\lt( \tz^* =\infty \rt)=0$ for all $x\in E$.
\end{thm}

\begin{proof}
The definition of $f^*$ given in (\ref{op-p}) and Definition \ref{s3-t-p} show that for each $\oz\in \ooz$
\begin{align*}
\tz_{f^*}^T(\oz)=&\inf\lt\{n\in \mathbb{N}|f^*(T-S(\oz),X_n(\oz))=1\rt\}\notag\\
=&\inf \lt\{n\in \mathbb{N}\bigg|U^*(T-\sum_{k=1}^n t_k,x_n)=g(x_n), T>\sum_{k=1}^n t_k \rt\}\notag\\
=&\inf\lt \{n\in \mathbb{N}\bigg| (T-\sum_{k=1}^n t_k,x_n)\in D^* \rt\},
\end{align*}
which, implies
$\tz^*(\oz)=\tau^T_{f^*}(\oz)\wedge \inf\bigg\{n\big|T\leq\sum_{k=1}^n t_k\bigg\}.$
By $V^*(0,x)=0=V^{\tz^*}(0,x)$, we only consider the case $T>0$. First, we have
\begin{align*}
&\{\tau^*=0\}=\{\tau^T_{f^*}=0\}\cup\{\tz^T_{f^*}>0,T\leq 0\}=\{\tau^T_{f^*}=0\} ;\notag\\
&\{\tau^*=n+1\}=\{\tau^T_{f^*}=n+1\}\cup\{\tz^T_{f^*}>n+1,S_n<T\leq S_{n+1}\}, \quad \forall n\geq 0 ;\notag\\
&\{\tau^*=\infty\}=\{\tau^T_{f^*}=\infty\}\cap\{T>\lim_{n \to \infty}S_n\}.
\end{align*}
Hence, using that $f^*$ is $T$-optimal,  Theorem \ref{s3-thm1} and Theorem \ref{ost-thm}, we obtain that
\begin{align*}
V^{\tau^*}(T,x)
=&\sum_{n=1}^{\infty}\lt\{\EE_x\lt[ \one_{\{\tau^T_{f^*}=n\}} R^T_n\rt]+\EE_x\lt[ \one_{\{S_{n-1}<T\leq S_n\}}\one_{\{\tz^T_{f^*}>n\}} R^T_n\rt]\rt\}+\EE_x\lt[ \one_{\{\tau^T_{f^*}=0\}} R^T_n\rt]\\
  &+\EE_x\lt[ \one_{\{\tau^T_{f^*}=\infty\}}\one_{\{T>\lim_{n \to \infty}S_n\}} \int_0^T c(X(t))\d t\rt]\\
=&\sum_{n=0}^{\infty}\hat{\EE}_x\lt[ \one_{\{\tau^T_{f^*}=n\}} R^T_n\rt]+\sum_{n=1}^{\infty}\EE_x\lt[ \one_{\{S_{n-1}<T\leq S_n\}} \one_{\{\tz^T_{f^*}=\infty\}}R^T_{\tau^T_{f^*}}\rt]\\ 
=&V^{{\tz^T_{f^*}}}(T,x)=V^*(T,x).
\end{align*}
Therefore, $\tz^*$ is $T$-optimal stopping time. Moreover, using (\ref{f-ost}), we have
\begin{align*}
0\leq \PP_x(\tz^*=\infty) \leq \PP_x(T>\lim_{n \to \infty}S_n)=0 ,\forall x\in E,
\end{align*}
which means $\tz^*$ is a finite stopping time, whereas $\tau^T_{f^*}$ does not necessarily.
\end{proof}

The condition of Theorem \ref{ost-thm} requires the value function $V^*$, but in practical applications, $V^*$ is often unknown. Intuitively, we can replace $V^*$ by the approximation function $V^*_n$, which is obtained by the iterative algorithm given in Theorem \ref{ost-thm}. Therefore, the concept of optimal stopping time will be replaced by $\varepsilon$-$T$-optimal, that is the following definition.

\begin{defn}
Given any $\varepsilon>0$, a stopping time $\tz$ is called $\varepsilon$-$T$-optimal
if it holds that $V^{\tz}(T,x) -V^*(T,x)\leq \varepsilon$ for all $x\in E$, where $V^{\tz}(T,x)$ is $T$-horizon expected cost of the stopping time $\tau $ given in (\ref{st-c}).
\end{defn}

The following theorem shows that for any $\varepsilon>0$, we can iterate enough times and get an $\varepsilon$-$T$-optimal stopping time under some conditions. For the convenience of statement, we give two notations, i.e. $||f|| := \sup_{x\in C} |f(x)|$ for any function $f$ defined on the set $C$; $\lceil x \rceil := \min \{ n \in \mathbb{N} : n \geq x \}$ for any $x\in \RR_+$.

\begin{thm}\lb{var-thm}
Suppose that $c$ and $g$ are bounded and that the semi-Markov kernel $Q$ satisfies $\sup_{x\in E}Q(T,E|x)=:\beta<1$. For any $\varepsilon>0$, the number of iterations $N_\varepsilon$ is given by
$$
N_{\varepsilon}:=\lt\lceil\frac{\log(\varepsilon(1-\beta))-\log(M+1)}{\log\beta} \rt\rceil,
$$
where $M=T||c||+||g||$. Let $V^*_{N_\varepsilon} (\lz,x)$ be the $N_\varepsilon$-th step iterative function given in Theorem \ref{ost-thm} and $D^\varepsilon$ be the subset of $[0,T]\times E$ given by
$$
D^\varepsilon:=\lt\{(\lz,x)\in (0,T]\times E|g(x)=\mathbb{G} V^*_{N_\varepsilon} (\lz,x)\rt\}.
$$
Then, the following statements hold.
\begin{itemize}
\item[ \rm (a)] Then, define the stopping time $\tau^\varepsilon$ by
\begin{equation*}
\tau^\varepsilon(\omega)=\inf\bigg\{n\big|((T-\sum_{k=1}^n t_k,x_n)\in D^\varepsilon \bigg\}\wedge \inf\bigg\{n\big|T\leq\sum_{k=1}^n t_k\bigg\}.
\end{equation*}
Then, $\tau^\varepsilon$ is an $\varepsilon$-$T$-optimal stopping time.
\item[\rm (b)] If it holds that
\be\lb{st-op}
\inf_{(\lz,x) \in ((0,T]\times E) \setminus D^\varepsilon} \lt( g(x)-\mathbb{G}V^*_{N_\varepsilon}(\lz,x) \rt) > \varepsilon,
\de
then, $\tau^\varepsilon$ is also the $T$-optimal stopping time.
\end{itemize}
\end{thm}

\begin{proof}
By (\ref{Q}), we have that
\begin{equation*}
\hat{Q}_T (T, \hat{E} | x, a) :=
\left\{
\begin{array}{ll}
	Q(T, E| x) \leq \beta, & x \in E, a = 0; \\
	\one_{ \lt[T+1, +\infty \rt) } (T)\delta_{\Delta}(\hat{E})= 0 \leq \beta, & x \in \hat{E},  a=1.
\end{array}
\right.
\end{equation*}
Then, using that $c$ and $g$ are bounded and $\mathbb{T}$ is a monotone operator, we have
$$||U_n^*||\leq ||U^*||\leq T||c|| +||g||=M,\quad \forall n\geq 0.$$
Moreover, by the definitions of $\mathbb{T}$, for each $(\lz,x)\in [0,T]\times E$, we have
\begin{align*}
 \mathbb{T} U_{n+1}^* (\lz,x)
\leq & \mathbb{T} U_{n}^* (\lz,x)+\beta||U^*_{n+1}-U^*_n||,
\end{align*}
which implies that
$
||U^*_{n+1}-U^*_{n}||= \beta ||U^*_{n}-U^*_{n-1}||\leq \beta^n ||U^*_{1}-U^*_{0}||\leq \beta^n M.
$
Then, we define a policy $f^\varepsilon$ by
$$
f^\varepsilon(\lz,x)=\one_{(0,\infty)\times E}(\lz,x)\one_{\{\TT U^*_{N_{\varepsilon}}(\lz,x)=\hat{g}(x,1)\}}(\lz,x)+\one_{\{\Delta\}}(x),
$$
and then, we have $f^\varepsilon\in \ppz_{HD}^0$. Moreover, similar to the proof of Theorem \ref{ost-thm}, we have that
\be\lb{t-f-n}
U^*_{N_{\varepsilon}+1}(\lz,x)=\TT U^*_{N_{\varepsilon}}(\lz,x)=\TT^{f^\varepsilon(\lz,x)} U^*_{N_{\varepsilon}}(\lz,x),\quad
\forall (\lz,x)\in[0,T]\times \hat{E}.
\de
Then, for each $n \geq 0$, by induction, we can show that
\be\lb{var-f}
U^*_{N_{\varepsilon}+1}(\lz,x)\geq U^{f^\varepsilon}_n(\lz,x)+\hat{\EE}_{(\lz,x)}^{f^\varepsilon}\lt[ U^*_{N_{\varepsilon}+1}(\lz-\hat{S}_{n+1},\hat{X}_{n+1})\rt]-\sum_{m=0}^n \beta^{m+1} ||U^*_{N_{\varepsilon}+1}-U^*_{N_{\varepsilon}}||.
\de
By the definition of $\TT^{f^\varepsilon(\lz,x)}$, we obtain that
\begin{align*}
U^*_{N_{\varepsilon}+1}(\lz,x)
= & U^{f^\varepsilon}_0(\lz,x)+\int_{[0,\lz]}\int_{\hat{E}} U^*_{N_{\varepsilon}}(\lz-t,y)\hat{Q}_T(\d t,\d y|x,f^\varepsilon(\lz,x))\\
\geq &  U^{f^\varepsilon}_0(\lz,x)+\hat{\EE}_{(\lz,x)}^{f^\varepsilon}\lt[ U^*_{N_{\varepsilon}+1}(\lz-\hat{S}_{1},\hat{X}_{1})\rt]-\beta||U^*_{N_{\varepsilon}+1}-U^*_{N_{\varepsilon}}||,
\end{align*}
which means that (\ref{var-f}) holds for $n=0$. On the other hand, for any $n\geq 0$
\begin{align*}
&\hat{\EE} \lt[U^*_{N_{\varepsilon}}(\lz-\hat{S}_{n+1},\hat{X}_{n+1})\rt]\\
=&\int_{\hat{E}} \delta_{x}(\d x_0)\int_{[0,s]}\int_{\hat{E}} \hat{Q}_T (\d t_1,\d x_1|x_0,f^\varepsilon(\lz,x_0))\int_{[0,\lz-s_1]}\int_{\hat{E}} \hat{Q}_T (\d t_2,\d x_2|x_1,f^\varepsilon(\lz- s_1,x_1))\\
& \cdots\int_{[0,\lz-s_n]}\int_{\hat{E}} \hat{Q}_T (\d t_{n+1},\d x_{n+1}|x_1,f^\varepsilon(\lz-s_n,x_n))U^*_{N_{\varepsilon}}(\lz-s_{n+1},x_{n+1})\\
\geq &\int_{\hat{E}} \delta_{x}(\d x_0)\int_0^\lz\int_{\hat{E}} \hat{Q}_T (\d t_1,\d x_1|x_0,f^\varepsilon(\lz,x_0))\int_{[0,\lz-s_1]}\int_{\hat{E}} \hat{Q}_T (\d t_2,\d x_2|x_1,f^\varepsilon(\lz- s_1,x_1))\\
& \cdots  \int_{[0,\lz-s_n]}\int_{\hat{E}} \hat{Q}_T (\d t_{n+1},\d x_{n+1}|x_1,f^\varepsilon(\lz-s_n,x_n))
 (U^*_{N_{\varepsilon}+1}(\lz-s_{n+1},x_{n+1})-||U^*_{N_{\varepsilon}+1}-U^*_{N_{\varepsilon}}||)\\
\geq &\hat{\EE} \lt[U^*_{N_{\varepsilon}+1}(\lz-\hat{S}_{n+1},\hat{X}_{n+1})\rt]-\beta^{n+1} ||U^*_{N_{\varepsilon}+1}-U^*_{N_{\varepsilon}}||.
\end{align*}
Thus, suppose that (\ref{var-f}) holds for some $n$, it holds that by (\ref{dpp-2}), (\ref{dpp-3}) and (\ref{t-f-n})
\begin{align*}
U^*_{N_{\varepsilon}+1}(\lz,x)\geq %&U^{f^\varepsilon}_n(\lz,x)+\hat{\EE}_{(\lz,x)}^{f^\varepsilon}\lt[ 
&U^{f^\varepsilon}_n(\lz,x)+\hat{\EE}_{(\lz,x)}^{f^\varepsilon}\lt[ \TT^{f^\varepsilon} U^*_{N_{\varepsilon}}(\lz-\hat{S}_{n+1},\hat{X}_{n+1})\rt]-\sum_{m=0}^n \beta^{m+1} ||U^*_{N_{\varepsilon}+1}-U^*_{N_{\varepsilon}}||\\
=&U^{f^\varepsilon}_{n+1}(\lz,x)+\hat{\EE}_{(\lz,x)}^{f^\varepsilon}\lt[  U^*_{N_{\varepsilon}}(\lz-\hat{S}_{n+2},\hat{X}_{n+2})\rt]-\sum_{m=0}^n \beta^{m+1} ||U^*_{N_{\varepsilon}+1}-U^*_{N_{\varepsilon}}||\\
\geq &U^{f^\varepsilon}_{n+1}(\lz,x)+\hat{\EE}_{(\lz,x)}^{f^\varepsilon}\lt[  U^*_{N_{\varepsilon}+1}(\lz-\hat{S}_{n+2},\hat{X}_{n+2})\rt]-\sum_{m=0}^{n+1} \beta^{m+1} ||U^*_{N_{\varepsilon}+1}-U^*_{N_{\varepsilon}}||.
\end{align*}
Hence, passing the limit $n \rightarrow \infty$ in (\ref{var-f}), it holds that for all $(\lz,x)\in [0,T]\times E$
\be\lb{s3-thm4-6}
U^*(\lz,x) \geq U^*_{N_\varepsilon+1} (\lz,x) \geq U^{f^\varepsilon}(\lz,x)- \frac{\beta}{1-\beta}||U^*_{N_\varepsilon+1}-U^*_{N_\varepsilon}||
\geq  U^{f^\varepsilon}(\lz,x)-\varepsilon.
\de
In the same method of Theorem \ref{ost-thm1}, we can verify
$\tz^\varepsilon=\tz^T_{f^\varepsilon}\wedge \inf\bigg\{n\big|T\leq S_n\bigg\}$
and
$$V^{\tz^\varepsilon}(T,x)=U^{f^{\varepsilon}}(T,x)\leq V^*(T,x)+\varepsilon,$$
i.e., $\tau^\varepsilon$ is an $\varepsilon$-$T$-optimal stopping time.

If $((0,T]\times E) \setminus D^\varepsilon= \varnothing$, then $(0,T]\times E = D^\varepsilon$ and the condition (\ref{st-op}) holds naturally. Hence, for each $(\lz,x) \in (0,T]\times E$, we have $g(x) =  V^*_{N_{\varepsilon}+1 }(\lz,x) \leq  V^* (\lz,x)$. That means $ D^* = D^\varepsilon$, where $D^*$ given in (\ref{st-set}).

Next, we consider the case $((0,T]\times E) \setminus D^\varepsilon \neq \varnothing$. Again, using the monotonicity of $\mathbb{G}$, we have $D^\vz \subset D^*$. Conversely, by the definition of $V^*_{N_\varepsilon} (\lz,x)$ and (\ref{s3-thm4-6}), it holds that
$$
\mathbb{H} V^*_{N_\varepsilon} (\lz,x) = V^*_{N_\varepsilon + 1} (\lz,x) \geq U^{f^\vz} (\lz,x) - \vz \geq U^* (\lz,x) - \vz = V^* (\lz,x) - \vz,
$$
i.e. $V^* (\lz,x) \leq \mathbb{H} V^*_{N_\varepsilon} (\lz,x) + \vz$ for each $(s,x) \in (0,T]\times E$. For each $(\lz,x) \in  ((0,T]\times E) \setminus D^\varepsilon$, the condition (\ref{st-op}) implies that
$$
g(x) -  V^*(\lz,x)  \geq g(x) - \mathbb{G} V^*_{N_\varepsilon}(\lz,x) - \varepsilon >0,
$$
which means $(\lz,x) \in ((0,T]\times E) \setminus D^*$. Hence, we have $D^* \subset D^\varepsilon$ and then $D^* = D^\varepsilon$. Finally, we have $\tau^\varepsilon = \tau^*$, which is a $T$-optimal stopping time given in Theorem \ref{ost-thm}.
\end{proof}

Hence, given any accuracy $\varepsilon>0$ and planning horizon $T>0$, we can devire an approach of computing $\varepsilon$-$T$-optimal stopping time.

{\bf An algorithm (for $\varepsilon$-$T$-optimal stopping time)}
\begin{itemize}
\item[\it Step 1](Initialization): Let $V^*_0(\lz,x)=0$ for every $(\lz,x)\in [0,T]\times E$.
\item[ \it Step 2](Iteration): Compute the function $V^*_n(\lz,x)$ for every $(\lz,x)\in [0,T]\times E$ by
$$V^*_{n+1}(\lz,x)=\min \bigg\{ c(x) \int_0^\lz  (1 - Q(t, E | x) )\d t + \int_{E}\int_{[0,\lz]}V^*_{n}(\lz-t,y) Q( \d t, \d y | x ), g(x) \bigg\}. $$
\item[ \it Step 3](Accuracy control): If $V^*_{n+1}(\lz,x)-V^*(\lz,x)\leq \varepsilon$ for every $(\lz,x)\in [0,T]\times E$, go to {\it Step 4}; otherwise, go to {\it Step 2} by replacing $n$ with $n+1$.
\item[ \it Step 4]($\varepsilon$-$T$-optimal stopping time): Compute the set 
$$D^\varepsilon=\lt\{(\lz,x)\in (0,T]\times E|g(x)=\mathbb{G} V^*_{N_\varepsilon} (\lz,x)\rt\}$$ and the $\varepsilon$-$T$-optimal stopping time
$$ \tau^\varepsilon(\omega)=\inf\{n\big|((T-\sum_{k=1}^n t_k,x_n)\in D^\varepsilon \}\wedge \inf\{n\big|T\leq\sum_{k=1}^n t_k\}, \  \forall \oz=(x_0,\ldots,x_n,t_{n+1},\ldots)\in \ooz. $$
\end{itemize}

%%%%%%%%%%Declarations%%%%%%%%%%

\fund % Place any funding information for this work after the \fund (or \Fund) command.
\noindent This work was partly supported by the National Natural Science Foundation of China (No. 11931018, 61773411, 11701588) and the Guangdong Basic and Applied Basic Research Foundation (No. 2020B1515310021).

%%%%%%%%%%%%Reference list%%%%%%%%%%%%%%
%
% References should be in the following form (or the BibTeX file
% apt.bst should be used):
%
% For a journal:
% Surname, Initial (year). Title of paper. {\em Journal title}
% {\bf Vol,} page--range.
%
% For a book:
% Surname, Initial (year). {\em Book title}. Publisher, Address.
%
% Note the following example of a reference list.

\end{document}